\documentclass{article}
\usepackage{eurosym}
\usepackage[T1]{fontenc}
\usepackage{float,color,ulem}
\usepackage{makeidx}
\usepackage{amssymb}
\usepackage{amsfonts}
\usepackage{amsmath}
\usepackage{hyperref}
\usepackage{caption}
\usepackage{graphicx,multirow}
\usepackage{epsf,epsfig,subfigure,diagbox}
\usepackage[toc,page,title,titletoc,header]{appendix}
\usepackage{indentfirst}
\hypersetup{urlcolor=blue,linkcolor=blue ,citecolor=blue,colorlinks=true}
\usepackage[english]{babel}
\usepackage{bbm, dsfont}
\newtheorem{theorem}{Theorem}[section]

\newtheorem{assumption}[theorem]{Assumption}

\newtheorem{definition}[theorem]{Definition}

\newtheorem{lemma}[theorem]{Lemma}

\newtheorem{remark}[theorem]{Remark}

\numberwithin{equation}{section}
\newenvironment{proof}[1][Proof]{\textbf{#1.} }{\ \rule{0.5em}{0.5em}}

\newcommand{\Li}{{\rm L}}

\newcommand{\W}{{\rm W}}
\newcommand{\loc}{{\rm loc}}

\begin{document}

\title{\textbf{Positively Invariant Subset for Non-Densely Defined Cauchy Problems}}

\author{Pierre Magal$^{{\rm (a)}}$ and Ousmane Seydi$^{{\rm (b)}}$  \\
{\small $^{{\rm (a)}}$ \textit{Univ. Bordeaux, IMB, UMR 5251, F-33400 Talence, France} } \\
{\small \textit{CNRS, IMB, UMR 5251, F-33400 Talence, France.}} \\
{\small $^{{\rm (b)}}$ \textit{D\'epartement Tronc Commun,\'Ecole Polytechnique de Thi\`es,
S\'en\'egal}} 
}

\maketitle
\begin{abstract}
Under weak Hille-Yosida conditions and using a generalized notion of sub tangential condition,  we prove the positive invariance of a closed subset by the semiflow generated by a semi-linear non densely defined Cauchy problem. A simple remark shows that the sufficient condition for the positivity of the semiflow implies our sub tangentiality  condition. But the sub tangential condition applies to a much larger class of closed positively invariant subset. Our results can be applied to hyperbolic and parabolic partial differential equations as well as functional differential equations.  As an illustration we apply our results to an age-structured equation in $L^p$ spaces which is only defined on a closed subset of $L^p$.
\end{abstract}
\vspace{0.2in}\noindent \textbf{Key words}. Semilinear differential equations, non-dense domain, integrated semigroup, positively invariant subset, age structured population dynamics models.

\vspace{0.1in}\noindent \textbf{AMS Subject Classification}. 37N25, 92D25, 92D30

%\tableofcontents 

\section{Introduction}

In this article we consider an abstract semi-linear Cauchy problem
\begin{equation}\label{1.1}
\dfrac{du(t)}{dt}=Au(t)+F(t,u(t)), \text{ for } t \geq 0, \text{ with } u(0)=u_0 \in \overline{D(A)},
\end{equation}
where $A:D(A) \subset X \to X$ is a linear operator on a Banach space $X$, and $F:[0, \infty) \times \overline{D(A)} \to X$ is continuous. We assume that the map $x \to F(t,x)$ is Lipschitz on the bounded sets of $\overline{D(A)}$ uniformly with respect to $t$ in a bounded interval of $[0, \infty)$. We point out that $D(A)$ is not necessarily dense in $X$ and $A$ is not necessarily a Hille-Yosida operator.

The invariance of subset for differential equation has a long history which starts with the seminal paper of the Japanese mathematician Nagumo \cite{Nagumo} in 1942. The result for ordinary differential equations was rediscovered later on by Brezis \cite{Brezis} and Hartman \cite{Hartman} and was further extended to ordinary differential equation in ordered Banach spaces by Walter \cite{Walter71} and Redheffer and Walter \cite{Redheffer-Walter75}. Several extensions to partial differential equations were proposed later on by Redheffer and Walter  \cite{Redheffer-Walter77} and Martin \cite{Martin79} for parabolic equations, etc. Martin and Smith \cite{Martinsmith} further investigated comparison/differential inequalities and invariant sets for abstract functional differential equations and reaction-diffusion systems that have time delays in the nonlinear reaction terms, and their developed results have had many applications. We refer to the book of Pavel and Motreanu \cite{Pavel-Motreanu} for an extensive study of densely defined semi-linear Cauchy problem. In \cite{Pavel-Motreanu} the authors studied the positive invariance for general closed subset subjected to tangency condition.  They also conidered positive invariance of time dependent closed subset and extended their results to semilinear differential inclusion problems. The case of closed convex subset for non-densely defined Cauchy problems with a Hille-Yosida linear operator perturbed by Lipschitz continuous non linear map has been studied by Thieme \cite{Thieme90a}. The goal of this article is to extend the results of Thieme \cite{Thieme90a} from the Hille-Yosida case to the non Hille-Yosida case. It is worth noting that the non Hille-Yosida case induces several difficulties due to the problem of non uniform boundedness of $\lambda (\lambda-A)^{-1}$ whenever $\lambda$ becomes. To overcome these difficulties we use a somewhat different approach compared to Thieme \cite{Thieme90a} combined with some generalization of the estimates on the integrated semigroup from the space of continuous functions to the space of regulated functions. Thank to our weak Hille-Yosida condition on the linear operator $A$ in \eqref{1.1} (see Assumption \ref{ASS4.4}) combined together with our generalized sub tangential condition (see Assumptions \eqref{ASS2.1} and \eqref{ASS2.4}) we can be applied our result to hyperbolic and parabolic partial differential equations in $L^p$ instead of $L^1$. 

%Several examples of differential equations, such as delay differential equations \cite{Ducrot-Magal-Ruan2013, Liu-Magal-Ruan}, parabolic equation with non-linear and non local boundary conditions \cite{Chu-Ducrot-Magal-Ruan, Ducrot-Magal-Prevost} can be put into the present framework of non densely abstract Cauchy problem \eqref{1.1}. More examples can be found in \cite{Magal-Ruan2018}. Thus, our developed results in this paper will have a wide range of applicability.% {\color{red} Let us mention that we plan to continue this work in order to generalize the results obtained in this paper. More precisely here we consider a non linearity $F$ for system \eqref{1.1} which is defined on the whole state space however it would be interesting to study the existence, the uniqueness, and the properties of the maximal autonomous or non autonomous semiflow  when $F$ is defined only on a sub-domain of the state space.} 

The paper is organized as follows. In sections 2 and 3 we recall some basic results about non densely defined Cauchy problems. In section 4, we investigate the positive invariance of a closed subset.  In section 5, we apply our result to an age-structured equation in $L^p$ spaces which only defined in a closed subset of $L^p$ and show that it generates a semiflow.
\section{Preliminary results}
Let $A: D(A) \subset X \to X$ be a linear operator. In the following we use the following notations
$$
X_0:=\overline{D(A)}
$$
and $A_0:D(A_0) \subset X_0 \to X_0$ the part of $A$ in $X_0$ that is
$$
A_0 x=A x, \quad \forall x\in D(A_0),
$$
and
$$
D(A_0)=\lbrace x\in D(A): Ax\in X_0\rbrace.
$$
\begin{assumption} \label{ASS2.1}
We assume that
\begin{itemize}
\item[{\rm (i)}] There exist two constants $\omega_A\in \mathbb{R}$ and $M_A\geq 1$,
such that $(\omega_A,+\infty)\subset \rho (A)$ and
\begin{equation*}
\left\Vert (\lambda I-A)^{-n}\right\Vert _{\mathcal{L}(X_0)}\leq
M_A \left( \lambda -\omega_A \right) ^{-n},\;\forall \lambda >\omega_A,\;\forall n \geq 1.
\end{equation*}
\item[{\rm(ii)}] $\lim_{\lambda \rightarrow +\infty }(\lambda I-A)^{-1}x=0,\
\forall x\in X$.
\end{itemize}
\end{assumption}
It is important to note that Assumption \ref{ASS2.1} does not say that $A$ is a Hille-Yosida linear operator since the operator norm in Assumption \ref{ASS2.1}-\textrm{(i)} is taken into $X_0 \subseteq  X$ (where the inclusion can be strict) instead of $X$. Further, it follows from \cite{Magal-Ruan2018} that $\rho(A)=\rho(A_0)$. Therefore by Assumption \ref{ASS2.1}, $(A_0,D(A_0))$ is a Hille-Yosida linear operator of type $(\omega_A,M_A)$ and generates a strongly continuous semigroup $\lbrace T_{A_0}(t) \rbrace_{t\geq 0}\subset \mathcal{L}(X_0)$ with
\begin{equation*}
\Vert T_{A_0}(t)\Vert_{\mathcal{L}(X_0)}\leq M_A e^{\omega_A t}, \quad \forall t\geq 0.
\end{equation*}
As a consequence
$$
\lim_{\lambda \to + \infty} \lambda \left(\lambda I -A \right)^{-1}x =x
$$
only for $x \in X_0$. It is important to note that the above limit does not exist in general whenever $x$ belongs to $X$.

We summarize the above discussions as follows.
\begin{lemma}\label{LE2.2} \ Assumption \ref{ASS2.1} is satisfied if and only if
there exist two constants, $M_A\geq 1$ and $\omega_A \in \mathbb{R},$
such that $\left( \omega_A ,+\infty \right) \subset \rho (A)$ and
$A_{0}$ is the infinitesimal generator of a $C_{0}$-semigroup
$\left\{ T_{A_{0}}(t)\right\} _{t\geq 0}$
on $X_{0}$ which satisfies $\left\| T_{A_{0}}(t)\right\| _{\mathcal{L}%
(X_{0})}\leq M_Ae^{\omega_A t},\forall t\geq 0$.
\end{lemma}

Next, we consider the non  homogeneous Cauchy problem
\begin{equation}\label{2.1}
v^\prime(t)=\ A v(t)+f(t), \quad t\geq0 \quad \text{and} \quad v(0)=v_0\in X_0,
\end{equation}
with $f\in L^{1}_{\loc}(\mathbb{R},X)$.

The integrated semi-group is one of the major tools to investigate non-homogeneous Cauchy problems. This notion was first introduced by Ardent \cite{Arendt87a, Arendt87b}.  We refer to the books Arendt et al. \cite{Arendt01} whenever $A$ an Hille-Yosida operator. We refer to Magal and Ruan \cite{Magal-Ruan07, Magal-Ruan2018} and Thieme \cite{Thieme08} for an integrated semi-group theory whenever $A$ is not Hille-Yosida operator. We also refer to the book of Magal and Ruan \cite{Magal-Ruan2018} for more references and results on this topic.

\begin{definition} \label{DE2.3} Let Assumption \ref{ASS2.1} be satisfied. Then $\left\lbrace S_A(t) \right\rbrace_{t \geq 0} \in \mathcal{L}(X)$ the \textbf{integrated semigroup generated by} $A$ is a strongly continuous family of bounded linear operators  on $X$, which is defined by
$$
S_A(t)x=(\lambda I-A_0) \int_0^t T_{A_{0}}(l) (\lambda I-A)^{-1}x dl, \forall t \geq 0.
$$
\end{definition}

In order to obtain existence and uniqueness of solutions for \eqref{2.1} whenever $f$ is a continuous map, we will require the following assumption.

\begin{assumption}\label{ASS2.4} Assume that for any $\tau >0$ and $f\in C\left( \left[ 0,\tau %
\right] ,X\right) $ there exists $v_{f}\in C\left( \left[
0,\tau \right] ,X_{0}\right) $ an integrated (mild) solution of
\begin{equation*}
\frac{dv_{f}(t)}{dt}=Av_{f}(t)+f(t),\text{ for }t\geq 0\text{ and }v_{f}(0)=0,
\end{equation*}
that is to say that
$$
\int_0^t v_f(r)dr \in D(A), \ \forall t\geq 0
$$
and
$$
v_f(t)=A\int_0^t v_f(r)dr +\int_0^t f(r)dr , \ \forall t\geq 0.
$$
Moreover we assume that there exists a non decreasing map $\delta
:[0,+\infty) \rightarrow [0,+\infty)$ such that
\begin{equation*}
\Vert v_f(t) \Vert \leq \delta(t) \underset{s\in [0,t]}{\sup} \Vert
f(s)\Vert , \ \forall t\geq 0,
\end{equation*}
with $\delta(t)\rightarrow 0$ as $t\rightarrow 0^+$.
\end{assumption}
\begin{remark} \label{REM2.5}
Note that Assumption \ref{ASS2.4} is equivalent (see \cite{Magal-Ruan09a}) to the assumption that there exists a non-decreasing map $\delta : [0,+\infty)\rightarrow [0,+\infty)$ such that for each $\tau >0$ and each $f\in C\left( \left[ 0,\tau \right] ,X\right)$ the map $t\rightarrow (S_A\ast f)(t)$ is differentiable in $[0,\tau]$ with
$$
\Vert (S_A \diamond f)(t) \Vert\leq \delta(t) \sup_{s\in [0,t]} \Vert f(s) \Vert,\ \forall t\in [0,\tau],
$$
where $(S_A\ast f)(t)$ and $(S_A \diamond f)(t)$ will be defined below in Theorem~\ref{TH2.7} and equation \eqref{2.3}.
\end{remark}
\begin{remark}\label{REM2.6}
It is important to point out the fact Assumption \ref{ASS2.4} is also equivalent to saying that $\left\lbrace S_A(t)\right\rbrace_{t\geq 0}\subset \mathcal{L}(X,X_0)$ is of bounded semi-variation on $[0,t]$ for any $t>0$ that is to say that
$$
V^\infty(S_A,0,t):=\sup\left\lbrace \left\Vert  \sum_{i=0}^{n-1} [S_A(t_{j+1})-S_A(t_{j})]x_j \right\Vert\right\rbrace <+\infty
$$
where the supremum is taken over all partitions $0=t_0<\dots<t_n=t$ of $[0,t]$ and all elements $x_1,\dots,x_n\in X$ with $\Vert x_j \Vert\leq 1$, for $j=1,2,\ldots,n$. 
Moreover the non-decreasing map $\delta : [0,+\infty)\rightarrow [0,+\infty)$ in Assumption \ref{ASS2.4} is defined by
$$
\delta(t):= \sup_{s \in [0,t]} V^\infty(S_A,0,s),\ \forall t\geq 0.
$$
\end{remark}

The following result is proved in \cite[Theorem 2.9]{Magal-Ruan09a}.
\begin{theorem}\label{TH2.7}  Let Assumptions \ref{ASS2.1} and \ref{ASS2.4} be satisfied. Then for each $\tau
>0$ and each $f\in C(\left[ 0,\tau \right] ,X)$ the map
$$
t\rightarrow \left(S_{A}\ast f\right) (t):=\int_{0}^{t}S_{A}(t-s)f(s)ds
$$ 
is continuously differentiable, $\left( S_{A}\ast
f\right) (t)\in D(A),\forall t\in \left[ 0,\tau \right] ,$ and if we set $%
u(t)=\frac{d}{dt}\left( S_{A}\ast f\right) (t),$ then
\begin{equation*}
u(t)=A\int_{0}^{t}u(s)ds+\int_{0}^{t}f(s)ds,\;\forall t\in \left[ 0,\tau %
\right] .
\end{equation*}
Moreover we have
\begin{equation*}
\left\| u(t)\right\| \leq \delta (t)\sup_{s\in \left[ 0,t\right] }\left\|
f(s)\right\| ,\;\forall t\in \left[ 0,\tau \right] .
\end{equation*}
Furthermore, for each $\lambda \in \left( \omega ,+\infty \right)$ we have
for each $t\in \left[ 0,\tau \right]$ that
\begin{equation} \label{2.2}
\left( \lambda I-A\right) ^{-1}\frac{d}{dt}\left( S_{A}\ast f\right)
(t)=\int_{0}^{t}T_{A_{0}}(t-s)\left( \lambda I-A\right) ^{-1}f(s)ds.
\end{equation}
\end{theorem}
From now on we will use the following notation
\begin{equation} \label{2.3}
\left( S_{A} \diamond f\right) (t):=\frac{d}{dt}\left( S_{A}\ast f\right) (t).
\end{equation}
From \eqref{2.2} and using the fact that $\left( S_{A} \diamond f\right) (t) \in X_0 $, we deduce the approximation formula
\begin{equation} \label{2.4}
\left( S_{A} \diamond f\right) (t)= \lim_{\lambda \to + \infty}\int_{0}^{t}T_{A_{0}}(t-s)\lambda \left( \lambda I-A\right) ^{-1}f(s)ds.
\end{equation}
A consequence of the approximation formula is the following
\begin{equation} \label{2.5}
\left( S_{A} \diamond f\right) (t+s)=T_{A_0}(s)\left( S_{A} \diamond f\right) (t)+ \left( S_{A} \diamond f(t+.)\right) (s), \forall t, s \geq 0.
\end{equation}
The following result is proved by Magal and Ruan \cite[Theorem 3.1]{Magal-Ruan07}, which will be constantly used and applied to the operator $A-\gamma B$ in sections 4 and 5.
\begin{theorem}[Bounded Linear Perturbation]~\\
\label{TH2.8}
Let Assumptions \ref{ASS2.1} and \ref{ASS2.4} be satisfied. Assume $L\in \mathcal{%
L}\left( X_{0},X\right) $ is a bounded linear operator.\ Then $%
A+L:D(A)\subset X\rightarrow X$ satisfies the conditions in Assumptions \ref{ASS2.1} and \ref{ASS2.4}. More precisely, if we fix $\tau _{L}>0$ such that
\begin{equation*}
\delta \left( \tau _{L}\right) \left\| L\right\| _{\mathcal{L}\left(
X_{0},X\right) }<1,
\end{equation*}
and if we denote by $\left\{ S_{A+L}(t)\right\} _{t\geq 0}$ the integrated
semigroup generated by $A+L,$ then for any $f\in C\left( \left[ 0,\tau _{L}%
\right] ,X\right)$, we have
\begin{equation*}
\left\| \frac{d}{dt}\left( S_{A+L}\ast f\right) \right\| \leq \frac{\delta
\left( t\right) }{1-\delta \left( \tau _{L}\right) \left\| L\right\| _{%
\mathcal{L}\left( X_{0},X\right) }}\sup_{s\in \left[ 0,t\right] }\left\|
f(s)\right\| ,\;\forall t\in \left[ 0,\tau _{L}\right] .
\end{equation*}
\end{theorem}
The following result is proved in \cite[Lemma 2.13]{Magal-Ruan09a}.
\begin{lemma}  Let Assumptions \ref{ASS2.1} and \ref{ASS2.4} be satisfied. Then
$$
\lim_{\lambda \to + \infty} \Vert \left( \lambda I -A \right)^{-1} \Vert_{\mathcal{L}(X)}=0.
$$
\end{lemma}
It follows that if $B \in \mathcal{L}(X_0,X)$, then for all $\lambda>0$ large enough the linear operator $\lambda I -A -B$ is invertible and its inverse can be written as follows
$$
\left(\lambda I -A -B \right)^{-1}=\left(\lambda I -A \right)^{-1}\left[ I-B\left(\lambda I -A \right)^{-1}\right]^{-1}.
$$
\section{Existence and Uniqueness of a Maximal Semiflow}
Consider now the non-autonomous semi-linear Cauchy problem
\begin{equation}\label{3.1}
U(t,s)x=x+A\int_{s}^{t}U(l,s)xdl+\int_{s}^{t}F(l,U(l,s)x)dl,\;\;\text{{}}%
t\geq s\geq 0,
\end{equation}
and the following problem
\begin{equation}
U(t,s)x=T_{A_{0}}(t-s)x+\frac{d}{dt}(S_{A}\ast F(.+s,U(.+s,s)x)(t-s),\text{ }%
t\geq s\geq 0.   \label{3.2}
\end{equation}%
We will make the following assumption.

\begin{assumption}  \label{ASS3.1} Assume that $F:\left[ 0,+\infty \right) \times \overline{D(A)}\rightarrow X$
is a continuous map such that for each $\tau _{0}>0$ and each $\xi >0,$
there exists $K(\tau _{0},\xi )>0$ such that
\begin{equation*}
\left\| F(t,x)-F(t,y)\right\| \leq K(\tau _{0},\xi )\left\| x-y\right\|
\end{equation*}
whenever $t\in \left[ 0,\tau _{0}\right] ,$\textit{\ }$y,x\in X_{0},$ and $%
\left\| x\right\| \leq \xi ,\left\| y\right\| \leq \xi .$
\end{assumption}
In the following definition $\tau $ is the blow-up time of maximal solutions
of (\ref{3.1}).
\begin{definition}[Non autonomous maximal semiflow]
\label{DE3.2} ~\\ Consider two maps $\tau :\left[ 0,+\infty \right)
\times X_{0}\rightarrow \left( 0,+\infty \right] $ and $U:D_{\tau
}\rightarrow X_{0},$ where
$$
D_{\tau }=\left\{ (t,s,x)\in \left[ 0,+\infty
\right) ^{2}\times X_{0}:s\leq t<s+\tau \left( s,x\right) \right\}.
$$
We say that $U$ is \textbf{ a maximal non-autonomous semiflow on} $X_{0}$ if $U$
satisfies the following properties
\begin{itemize}
\item[{\rm (i)}] $\tau \left( r,U(r,s)x\right) +r=\tau \left( s,x\right)
+s,\forall s\geq 0,\forall x\in X_{0},\forall r\in \left[ s,s+\tau \left(
s,x\right) \right)$.

\item[{\rm (ii)}] $U(s,s)x=x,\forall s\geq 0,\forall x\in X_{0}$.

\item[{\rm (iii)}]$U(t,r)U(r,s)x=U(t,s)x,\forall s\geq 0,\forall x\in
X_{0},\forall t,r\in \left[ s,s+\tau \left( s,x\right) \right) $ with $t\geq
r.$

\item[{\rm (iv)}] If $\tau \left( s,x\right) <+\infty ,$ then
\begin{equation*}
\lim_{t\rightarrow \left( s+\tau \left( s,x\right) \right) ^{-}}\left\|
U(t,s)x\right\| =+\infty .
\end{equation*}
\end{itemize}
\end{definition}

Set
\begin{equation*}
D=\left\{ \left( t,s,x\right) \in \left[ 0,+\infty \right) ^{2}\times
X_{0}:t\geq s\right\} .
\end{equation*}
The following theorem is the main result in this section, which was proved
in \cite[Theorem 5.2]{Magal-Ruan07}.
\begin{theorem} \label{TH3.3}
Let Assumptions \ref{ASS2.1}, \ref{ASS2.4} and \ref{ASS3.1} be satisfied. Then there exists a map $\tau
:\left[ 0,+\infty \right) \times X_{0}\rightarrow \left( 0,+\infty \right] $
and a maximal non-autonomous semiflow $U:D_{\tau }\rightarrow X_{0},$ such
that for each $x\in X_{0}$ and each $s\geq 0,$ $U(.,s)x\in C\left( \left[
s,s+\tau \left( s,x\right) \right) ,X_{0}\right) $ is a unique maximal
solution of (\ref{3.1}) (or equivalently a unique maximal solution of (\ref{3.2})). Moreover, $D_{\tau }$ is open in $D$ and the map $\left(
t,s,x\right) \rightarrow U(t,s)x$ is continuous from $D_{\tau }$ into $%
X_{0}. $
\end{theorem}

\section{Positive invariance of a closed subset}
In this section we will study the positive invariance  of a closed subset by imposing the so called sub-tangential condition. Our results extend those in \cite{Pavel-Motreanu,Thieme90a} since we focus on the study of non densely defined non Hille-Yosida semilinear Cauchy problems.  We start with some lemmas that will be useful in the subsequent discussions.
\begin{lemma}\label{LE4.1}
Let Assumptions \ref{ASS2.1} and \ref{ASS2.4} be satisfied. Let $0\leq a<b$ and $x \in X$  be given and define
$$
f(t):=x \mathbbm{1}_{[a,b]}(t),\ \forall t\geq 0.
$$
Then $t \rightarrow (S_A \ast f)(t)$ is differentiable in $[0,+\infty)$ and
$$
(S_A \diamond f)(t)=\dfrac{d}{dt}(S_A \ast f)(t)=S_A((t-a)^+)x-S_A((t-b)^+)x,\ \forall t\geq 0,
$$
where $\sigma^+:=\max(0,\sigma), \forall \sigma \in \mathbb{R}$.
\end{lemma}
\begin{proof}
We observe that
$$
(S_A \ast f)(t)=\left\lbrace
\begin{array}{lll}
 \int_{a}^t S_A(t-s)x ds & \text{ if } & t\in [a,b],\\
\int_{a}^b S_A(t-s)x ds & \text{ if } & t\geq b, \\
0 & \text{ if } & 0\leq  t\leq a,
\end{array}
\right.
$$
which is equivalent to
$$
(S_A \ast f)(t)=\left\lbrace
\begin{array}{lll}
\int_{0}^{t-a} S_A(s)x ds & \text{ if } & t\in [a,b],\\
\int_{t-b}^{t-a} S_A(s)x ds & \text{ if } & t\geq b, \\
0 & \text{ if } & 0\leq t\leq a.
\end{array}
\right.
$$
Then the formula follows by computing the time derivative.
\end{proof}\\
By using similar arguments in the proof of Lemma \ref{LE4.1} one can easily obtain the following results.
\begin{lemma}\label{LE4.2}
Let Assumptions \ref{ASS2.1} and \ref{ASS2.4} be satisfied. Let $0\leq a<b$ be given. Let $a=t_0<\dots < t_n=b$ be a partition of $[a,b]$. Let $f :[a,b]\rightarrow X$ be the step function defined by
$$
f(t):=\sum_{i=0}^{n-1} x_i \mathbbm{1}_{[t_i,t_{i+1})}(t),\ \forall  t\in[a,b) \ \text{ and } \ f(b)=f(t_{n-1})=x_{n-1}.
$$
Then $t \rightarrow (S_A \ast f(a+\cdot))(t-a)$ is differentiable in $[a,b]$ and  for any $t \in [t_k,t_{k+1}],\ k=0,\dots,n-1$ one has
$$
(S_A \diamond f(a+\cdot))(t-a)=\sum_{i=0}^{k-1}  [S_A(t-t_i)-S_A(t-t_{i+1})]x_i+S_A(t-t_k)x_k.
$$
\end{lemma}

Recall that $f:[a,b]\rightarrow X$ is a regulated function if the limit from the right side $\displaystyle \lim_{s\rightarrow t^+}f(s)$ exists for each $t \in [a,b)$, and the limit from the left side $\displaystyle \lim_{s\rightarrow t^-}f(s)$ exists for each $t\in (a,b]$.
For each $b>a\geq 0$, we assume ${\rm Reg}([a,b],X)$ denotes the space of regulated functions from $[a,b]$ to $X$, and we also denote by ${\rm Step}([a,b],X)$ the space of step functions from $[a,b]$ to $X$.

The following lemma extend the property described in Assumption \ref{ASS2.4} for the space of continuous functions to the space of regulated functions.
\begin{lemma}\label{LE4.3}
Let Assumptions \ref{ASS2.1} and \ref{ASS2.4} be satisfied and let $0\leq a<b$ be given. Then for any $f \in {\rm Reg}([a,b],X)$ we have
$$
\Vert (S_A \diamond f(a+\cdot))(t-a) \Vert \leq \delta(t-a) \sup_{s\in [a,t]} \Vert f(s) \Vert, \ \forall t\in [a,b].
$$
\end{lemma}
\begin{proof}
Since ${\rm Step}([a,b],X)$ is dense in ${\rm Reg}([a,b],X)$ for the topology of uniform convergence (see Dieudonne \cite[p.139]{Dieudonne}), it is sufficient to prove the result for $f \in {\rm Step}([a,b],X)$ and apply the linear extension theorem to the  bounded linear operator
$$
f \in {\rm Step}([a,b],X)  \mapsto (S_A \diamond f)(\cdot).
$$
Let $f \in {\rm Step}([a,b],X)$ be a non zero step function given by
$$
f(t):=\sum_{i=0}^{n-1} x_i \mathbbm{1}_{[t_i,t_{i+1})}(t),\ \forall t\in  [a,b),\ \text{ and } f(b)=f(t_{n-1})=x_{n-1}
$$
with $a=t_0<t_1<\cdots<t_n=b$. Let $t\in [a,b]$ be given and fixed. Then there exists $k\in \{0,\dots,n-1\}$ such that $t\in [t_k,t_{k+1}]$. Hence by Lemma \ref{LE4.2} we have
$$
\begin{array}{llll}
(S_A \diamond f(a+\cdot))(t-a)&=& \displaystyle \sum_{i=0}^{k-1}  [S_A(t-t_i)-S_A(t-t_{i+1})]x_i+S_A(t-t_k)x_k\\
  &=&  \displaystyle \sum_{i=0}^{k}  S_A(t-t_{k-i})x_{k-i}-\sum_{i=1}^{k}S_A(t-t_{k-i+1}) x_{k-i}\\
\end{array}
$$
Setting
$$
\bar{t}_i=t-t_{k-i+1},\  i=1,\dots, k \ \text{ and } \ \bar{t}_{0}:=0
$$
and
$$
\bar{x}_i:=\dfrac{x_{k-i}}{\alpha}, \ i=0,\dots, k
$$
with $\alpha:=\max_{i=1,\dots,k} \Vert x_i \Vert>0$. Then we obtain
$$
\begin{array}{llll}
(S_A \diamond f(a+\cdot))(t-a) &=& \alpha \displaystyle \sum_{i=0}^{k}  [S_A(\bar{t}_{i+1})-S_A(\bar{t}_{i})]\bar{x}_i.\\
\end{array}
$$
Since $0=\bar{t}_0<\cdots< \bar{t}_{k+1}=t-a$ and $\Vert \bar{x}_i \Vert\leq 1$ for all $i=1,\dots,k$, it follows from Remark \ref{REM2.6} that
$$
\Vert (S_A \diamond f(a+\cdot))(t-a)\Vert  \leq \alpha V^{\infty }(S_A,0,t-a)\leq \alpha \delta(t-a)
$$
and the result follows by observing that
$$
\alpha:=\max_{i=1,\dots,k} \Vert x_i \Vert=\sup_{s\in [a,t]}\Vert f(s) \Vert.
$$
\end{proof}

In order to prove the invariance property of a closed subset $C_0 \subset X_0$ we need to make the following assumption.
\begin{assumption}[Sub-Tangential Condition]\label{ASS4.4} Let $C_0$ be a closed subset of $X_0$.  We assume that there exists a bounded linear operator $B : X_0\rightarrow X$ such that for each  $\xi >0$ and each $\sigma>0$ there exists $\gamma=\gamma(\xi,\sigma)>0$ such that
$$
\lim_{h\rightarrow 0^+}\dfrac{1}{h} d\left( T_{(A-\gamma B)_0}(h)x+S_{A-\gamma B}(h)
\left[F(t,x)+\gamma Bx \right],C_0\right)=0,
$$
whenever $x \in C_0$ with  $ \Vert x \Vert \leq \xi $ and $t \in [0, \sigma]$. Here the map  $x \to d(x,C_0)$ is the Hausdorff semi-distance which is defined as
$$
d(x,C_0):=\inf_{y\in C_0} \Vert x-y \Vert.
$$
\end{assumption}
\begin{remark}
Recall that the usual assumption for the non negativity of the mild solutions of \eqref{1.1} is covered by Assumption \ref{ASS4.4}. In fact $X_{0+}$ is positively invariant with respect to semiflow generated by \eqref{1.1} if  for each $\xi >0$ and each $\sigma>0$ there exists $\gamma=\gamma(\xi,\sigma)>0$ such that
$$
T_{(A-\gamma B)_0}(h)x+S_{A-\gamma B}(h)[F(t,x)+\gamma Bx] \in X_{0+}
$$
whenever $x \in X_{0+}$ with  $ \Vert x \Vert \leq \xi $ and $t \in [0, \sigma]$.
\end{remark}
The main result of this article is the following theorem. 
\begin{theorem}[Positive invariant Subset] \label{TH4.6}
Let Assumptions \ref{ASS2.1}, \ref{ASS2.4}, \ref{ASS3.1} and \ref{ASS4.4} be satisfied. Then for each $x\in C_{0}$ and each $s\geq 0$,  we have
$$
U(t,s)x\in C_0, \forall t \in  \left[
s,s+\tau \left( s,x\right) \right).
$$
\end{theorem}
The rest of this section is devoted to the proof of Theorem \ref{TH4.6}. We fix the initial condition $x_0 \in C_0$ and $s=0$. Set $\rho:= 2 (\Vert x_0 \Vert+1)$ and  define
$$
F_\gamma(t,x):=F(t,x)+\gamma Bx, \ \forall (t,x) \in [0,+\infty)\times X_0.
$$
Let $\Lambda:= \Lambda(\rho)>0$ be the constant  such that
\begin{equation} \label{4.1}
\Vert F_\gamma(t,x)-F_\gamma(t,y)\Vert \leq \Lambda \Vert x-y \Vert, \forall t\in [0,\rho],\ \forall x,y \in B(0,\rho).
\end{equation}
Therefore by setting
 $$
 \Gamma:=2\Lambda \rho+\sup_{t\in [0,\rho]}\Vert F_\gamma(t,x_0) \Vert,
 $$
 we  obtain
\begin{equation}\label{4.2}
 \Vert F_\gamma(t,x) \Vert \leq \Gamma,\  \forall t\in [0,\rho],\  \forall x\in B(0,\rho).
\end{equation}
Let $\gamma:=\gamma(\rho)>0$ be a constant such that
\begin{equation}\label{4.3}
\lim_{h\rightarrow 0^+} \dfrac{1}{h} d\left( T_{(A-\gamma B)_0}(h)x+S_{A-\gamma B}(h)F_\gamma(t,x),C_0\right)=0,
\end{equation}
whenever $x \in C_0$,  $ \Vert x \Vert \leq \rho $ and $t \in [0, \rho]$.

Then by Theorem \ref{TH2.8}, $A-\gamma B: D(A)\subset X\rightarrow X$ satisfies Assumptions \ref{ASS2.1} and \ref{ASS2.4}. Hence combining Theorem \ref{TH2.8} and Lemma \ref{LE4.3} we know that if we fix $\tau _{\gamma}>0$ such that
\begin{equation*}
\gamma \delta \left( \tau _{\gamma}\right) \left\| B\right\| _{\mathcal{L}\left(
X_{0},X\right) }<1,
\end{equation*}
then there exists a non decreasing map $\delta_\gamma : [0,+\infty)\rightarrow [0,+\infty)$  with
$$\lim_{t\rightarrow 0^+}\delta_\gamma(t)=0$$ such that for each $f \in {\rm Reg}([a,b],X)$, $0\leq a<b\leq \tau_\gamma$
\begin{equation}\label{4.4}
\left\| \left( S_{A-\gamma B}\diamond f(a+\cdot)\right)(t-a) \right\| \leq \delta_\gamma(t-a)\sup_{s\in \left[ a,t\right] }\left\|
f(s)\right\| ,\;\forall t\in \left[ a,b\right] .
\end{equation}
To shorten the notations we set
$$
\omega_\gamma:=\omega_{_{A-\gamma B}} \ \text{ and } \ M_\gamma:= M_{_{A-\gamma B}}.
$$
 Let $\tau\in (0,\min(\tau \left( 0,x\right),\tau_\gamma,\rho))$ be small enough to satisfy 
\begin{equation}\label{4.5}
\Gamma \delta_\gamma(h)+M_{\gamma} e^{\omega_{\gamma}^+ h} h+\Vert T_{(_{A-\gamma B})_0}(h)x_0\Vert\leq \rho,\ \forall h\in [0,\tau]
\end{equation}
with
$$
\omega_\gamma^+=\max(0,\omega_\gamma),
$$
and 
\begin{equation}\label{4.6}
0<\Lambda \delta_\gamma (\tau)<1,
\end{equation}
where $\Lambda$ has been defined as an upper bound for the Lipschitz norm of $F_\gamma$ on $B(0,\rho)\cap C_0$ in \eqref{4.1}. 

\medskip

\noindent \textbf{Construction of the knots :} Let $ \varepsilon \in (0,1)$ be fixed.  We define by induction a sequence $(l_k,y_k) \in [0,\tau]\times C_0$ where the index $k \in \mathbb{N}$ is a non-negative integer possibly unbounded.  For $k=0$ we start with
$$
l_0= 0 \ \text{ and } y_0=x_0\in C_0.
$$
In order to compute the next increment, we define for each integer $k \geq 0$ 
\begin{equation}\label{4.7}
 \begin{array}{ll}
I_k = \left \lbrace \eta \in (0,\varepsilon^*) : \right. &
  \Vert F_\gamma(l,y)-F_\gamma(l_k,y_k)\Vert\leq \varepsilon,\  \forall \vert l-l_k\vert\leq \eta,  \ \forall y\in B(y_k,\eta)\cap C_0,\\
  &  \left.\dfrac{1}{\eta} d\left( T_{(A-\gamma B)_0}(\eta)y_k+S_{A-\gamma B}(\eta)F_\gamma(l_k,y_k),C_0\right) <\dfrac{\varepsilon}{2} \right. \\
  &  \left. \text{ and }  \Vert T_{(A-\gamma B)_0}(\eta)y_k-y_k \Vert  \leq \varepsilon\right\rbrace
 \end{array}
\end{equation}
where $\varepsilon^*:=\min(\varepsilon,\rho)$.

Set
\begin{equation}\label{4.8}
 r_k:= \sup(I_k)>0 \ \text{ and } l_{k+1}:=\min \left( l_k+\dfrac{r_k}{2},\tau\right).
\end{equation}
We define  
$$
y_{k+1}= y_k \in C_0 \text{ if } l_{k+1}=\tau. 
$$
Otherwise if $l_{k+1}=l_k+\dfrac{r_k}{2}<\tau$, then 
$$
0< l_{k+1}-l_{k}=\dfrac{r_k}{2}<r_k
$$ 
hence  
$$
l_{k+1}-l_k \in I_k.
$$
Thus, it follows that
$$
\dfrac{1}{l_{k+1}-l_k} d\left( T_{(A-\gamma B)_0}(l_{k+1}-l_k)y_k+S_{A-\gamma B}(l_{k+1}-l_k)F_\gamma(l_k,y_k),C_0\right)<\dfrac{\varepsilon}{2}.
$$
Therefore, we can find  $y_{k+1} \in C_0$ satisfying 
\begin{equation}\label{4.9}
\dfrac{1}{l_{k+1}-l_k}\Vert  T_{(A-\gamma B)_0}(l_{k+1}-l_k)y_k+S_{A-\gamma B}(l_{k+1}-l_k)F_\gamma(l_k,y_k)-y_{k+1}\Vert \leq  \dfrac{\varepsilon}{2}.
\end{equation}
Setting
$$
H_k:=\dfrac{1}{l_{k+1}-l_k} \left[ y_{k+1}-T_{(A-\gamma B)_0}(l_{k+1}-l_k)y_k-S_{A-\gamma B}(l_{k+1}-l_k)F_\gamma(l_k,y_k) \right] \in X_0.
$$
Then it follows that
\begin{equation}\label{4.10}
H_k\in X_0 \text{ and } \Vert H_k \Vert \leq \dfrac{\varepsilon}{2}
\end{equation}
and
\begin{equation}\label{4.11}
y_{k+1}=T_{(A-\gamma B)_0}(l_{k+1}-l_k)y_k+S_{A-\gamma B}(l_{k+1}-l_k)F_\gamma(l_k,y_k)+(l_{k+1}-l_k)H_k\in C_0.
\end{equation}
\begin{lemma}\label{LE4.7}
Let Assumptions \ref{ASS2.1}, \ref{ASS2.4},  \ref{ASS3.1} and \ref{ASS4.4} be  satisfied. Then the knots $(l_k,y_k)$, $k\geq 0$ satisfy the following properties
\begin{itemize}
\item[{\rm (i)}] For all $k>m\geq 0$ we have
\begin{equation}\label{4.12}
\begin{array}{lll}
 y_k&=& \displaystyle T_{(A-\gamma B)_0}(l_{k}-l_{m})y_m+\sum_{i=m}^{k-1} (l_{i+1}-l_{i})T_{(A-\gamma B)_0}(l_{k}-l_{i+1})H_i\\
 & & \displaystyle+\sum_{i=m}^{k-1} T_{(A-\gamma B)_0}(l_{k}-l_{i+1})S_{A-\gamma B}(l_{i+1}-l_i)F_\gamma(l_i,y_i)
\end{array}
\end{equation}
\item[{\rm (ii)}]  $y_k \in B(0,\rho)\cap C_0$ for any $k\geq 0$.
\item[{\rm (iii)}] For all $k>m\geq 0$ we have
$$
\Vert y_k-T_{(A-\gamma B)_0}(l_{k}-l_{m})y_m\Vert \leq \Gamma \delta_\gamma(l_k-l_m)+\dfrac{\varepsilon}{2} M_\gamma e^{\omega_\gamma^+ (l_k-l_m)} (l_k-l_m).
$$
%with $\omega_\gamma^+=\max(0,\omega_\gamma)$.
\end{itemize}
\end{lemma}
\begin{proof} \textbf{Proof of (i):} Let $k>m\geq 0$ be given.  Recall that for all $i=0,\dots,k-1$ we have
$$
 y_{i+1}=T_{(A-\gamma B)_0}(l_{i+1}-l_{i})y_{i}+S_{A-\gamma B}(l_{i+1}-l_{i})F_\gamma(l_{i},y_{i})+(l_{i+1}-l_{i})H_i.
 $$
Define the linear operator $L_i : X_0 \rightarrow X_0$ by
 $$
 L_i:=T_{(A-\gamma B)_0}(l_{i+1}-l_{i}),\ i=0,\dots,k-1
 $$
Hence
 $$
 y_{i+1}=L_iy_i+S_{A-\gamma B}(l_{i+1}-l_{i})F_\gamma(l_{i},y_{i})+(l_{i+1}-l_{i})H_i,\ i=1,\dots,k-1.
 $$
In order to use a variation of constants formula, we introduce the evolution family
 $$
 U(i,j)=L_{i-1}\cdots L_j \ \text{ if } i>j \ \text{ and } U(i,i)=I_{X_0}.
 $$
Then it follows from the semigroup property that 
 $$
 U(i,j)=T_{(A-\gamma B)_0}(l_{i}-l_{j}),\ \text{ if } i\geq j .
 $$
By using a discrete variation of constants formula, we have for integers $ k \geq m \geq 0$ 
\begin{equation*}
\begin{array}{lll}
 y_k&=&\displaystyle  U(k,m)y_m+\sum_{i=m}^{k-1} U(k,i+1)[S_{A-\gamma B}(l_{i+1}-l_{i})F_\gamma(l_{i},y_{i})+(l_{i+1}-l_{i})H_i] \\
 & =& \displaystyle T_{(A-\gamma B)_0}(l_{k}-l_{m})y_m+\sum_{i=m}^{k-1} (l_{i+1}-l_{i}) T_{(A-\gamma B)_0}(l_{k}-l_{i+1})H_i\\
 & &+ \displaystyle \sum_{i=m}^{k-1} T_{(A-\gamma B)_0}(l_{k}-l_{i+1})S_{A-\gamma B}(l_{i+1}-l_{i})F_\gamma(l_{i},y_{i}).
 \end{array}
\end{equation*}
\textbf{Proof of (ii):} We will argue by recurrence. The property is true for $k=0$ since $y_0=x_0 \in B(0,\rho)\cap C_0$. Assume that for $k\geq 1$
$$
y_0,\dots, y_{k-1} \in  B(0,\rho)\cap C_0.
$$
We are in a position to show that $y_k \in B(0,\rho)\cap C_0$. In view of \eqref{4.12}, for any $m=0,\dots,k-1$, we have
 $$
y_k-T_{(A-\gamma B)_0}(l_{k}-l_{m})y_m=\sum_{i=m}^{k-1} T_{(A-\gamma B)_0}(l_{k}-l_{i+1})[S_{A-\gamma B}(l_{i+1}-l_{i})F_\gamma(l_{i},y_{i})+(l_{i+1}-l_{i})H_i].
 $$
Then it follows that
$$
\Vert y_k-T_{(A-\gamma B)_0}(l_{k}-l_{m})y_m \Vert\leq \Vert W_{k,m} \Vert+\Vert Z_{k,m} \Vert,
$$
where
$$
W_{k,m}:=\sum_{i=m}^{k-1} T_{(A-\gamma B)_0}(l_{k}-l_{i+1})S_{A-\gamma B}(l_{i+1}-l_{i})F_\gamma(l_{i},y_{i})
$$
and
$$
Z_{k,m}:=\sum_{i=m}^{k-1} (l_{i+1}-l_{i})T_{(A-\gamma B)_0}(l_{k}-l_{i+1})H_i.
$$
Next, we do estimates of $W_{k,m}$ and $Z_{k,m}$. Since $H_i \in X_0$ and $\Vert H_i \Vert \leq \dfrac{\varepsilon}{2}$, for any $i=m,\dots,k-1$, it is easy to obtain from \eqref{4.10} that
\begin{equation}\label{4.13}
\begin{array}{lll}
 \Vert Z_{k,m}\Vert &\leq &  \displaystyle \sum_{i=m}^{k-1} (l_{i+1}-l_{i})\dfrac{\varepsilon}{2} M_\gamma e^{\omega_\gamma (l_{i+1}-l_{i})}\vspace{0.2cm}\\
  &\leq & \displaystyle\dfrac{\varepsilon}{2} M_\gamma e^{\omega_\gamma^+ (l_k-l_m)} (l_k-l_m),
\end{array}
\end{equation}
where
 $$
 \omega_\gamma^+=\max(0,\omega_\gamma).
 $$
In order to estimate $W_{k,m}$, we will rewrite it in a more convenient form. Using the following relationship
 $$
T_{(A-\gamma B)_0}(\sigma)S_{A-\gamma B}(h)=S_{A-\gamma B}(\sigma+h)-S_{A-\gamma B}(\sigma),\ \forall \sigma\geq 0,\forall h\geq 0,
$$
we see that
 $$
W_{k,m}= \sum_{i=m}^{k-1} [ S_{A-\gamma B}(l_{k}-l_{i})-S_{A-\gamma B}(l_{k}-l_{i+1}) ]F_\gamma(l_{i},y_{i}).
 $$
 By Lemma \ref{LE4.2} we have 
 $$
 W_{k,m}= (S_{A-\gamma B} \diamond f_\gamma(l_m+\cdot))(l_k-l_m)
 $$
with step function
 $$
 f_\gamma(t)=F_\gamma(l_{i},y_{i}),\ \forall t\in [l_i,l_{i+1}),\ i=m,\dots,k-1\ \text{ and } \ f_\gamma(l_k)=F_\gamma(l_{k-1},y_{k-1}).
 $$
Therefore by using the inequality \eqref{4.4} with $a=l_m$ and $b=l_k$  it follows that
\begin{equation}\label{4.14}
\begin{array}{lll}
\displaystyle  \left \Vert W_{k,m} \right \Vert  &= & \displaystyle  \left \Vert(S_{A-\gamma B} \diamond f_\gamma(l_m+\cdot))(l_k-l_m) \right \Vert.\\
&\leq & \displaystyle   \delta_\gamma(l_k-l_m) \sup_{s \in [l_m,l_k]} \Vert f_\gamma(s) \Vert \\
&=& \displaystyle   \delta_\gamma(l_k-l_m) \max_{i=m,\dots,k-1} \Vert F_\gamma(l_{i},y_{i}) \Vert.
 \end{array}
\end{equation}
By using \eqref{4.2} and the induction assumption, we deduce that 
$$
\max_{i=m,\dots,k-1} \Vert F_\gamma(l_{i},y_{i}) \Vert\leq \Gamma.
 $$
Then it follows from \eqref{4.13} and \eqref{4.14} that
\begin{equation*}
\Vert y_k-T_{(A-\gamma B)_0}(l_{k}-l_{m})y_m \Vert \leq \Gamma \delta_\gamma(l_k-l_m)+\dfrac{\varepsilon}{2} M_\gamma e^{\omega_\gamma^+ (l_k-l_m)} (l_k-l_m)
\end{equation*}
for $m=0,\dots,k-1$.
To conclude the proof of (ii) we note that
$$
\begin{array}{llll}
\Vert y_k-T_{(A-\gamma B)_0}(l_{k}-l_{0})x_0 \Vert=\Vert y_k-T_{(A-\gamma B)_0}(l_{k})x_0 \Vert &\leq &  \Gamma \delta_\gamma(l_k)+M_\gamma e^{\omega_\gamma^+ l_k} l_k \\
\end{array}
$$
and
$$
\begin{array}{lll}
\Vert y_k \Vert &\leq & \Vert y_k-T_{(A-\gamma B)_0}(l_{k})x_0 \Vert+\Vert T_{(A-\gamma B)_0}(l_{k})x_0 \Vert \\
& \leq &  \Gamma \delta_\gamma(l_k)+M_\gamma e^{\omega_\gamma^+ l_k} l_k+\Vert T_{(A-\gamma B)_0}(l_{k})x_0\Vert.
\end{array}
$$
Since $l_k\in [0,\tau]$, the inequality \eqref{4.5}  implies that $y_k \in B(0,\rho)\cap C_0$.

\medskip
\noindent \textbf{Proof of (iii): } The proof follows the same lines in (ii).
\end{proof}
\begin{lemma} \label{LE4.8}
Let Assumptions \ref{ASS2.1}, \ref{ASS2.4},  \ref{ASS3.1} and \ref{ASS4.4}  be  satisfied. Then there exists an integer $n_\varepsilon\geq 1$ such that $l_{n_\varepsilon}=\tau$. That is to say that we have a finite number of knots $(l_k,y_k)$, $k=0,\dots,n_\varepsilon$ with
$$
0=l_0<l_1<\cdots<l_{n_\varepsilon-1}< l_{n_\varepsilon}=\tau \ \text{ and }\  y_0,y_1,\dots,y_{n_\varepsilon} \in C_0,\ y_0=x_0.
$$
\end{lemma}
\begin{proof}
We will use proof by contradiction. Assume that $l_k<\tau$ for all $k\geq 0$. That is to say that
$$
l_{k+1}=l_k+\dfrac{r_k}{2}, \ \forall k\geq 0.
$$
Since the sequence is strictly increasing, there exists $l^*\leq \tau$ such that $l_k\rightarrow l^*$ as $k\rightarrow +\infty$ and $l_k<l^*$ for each $k\geq 0$. This also implies that 
\begin{equation}\label{4.15}
\lim_{k\rightarrow +\infty} r_k=0.
\end{equation}
In order to contradict \eqref{4.15}, we will prove that there exists $k_0$ large enough and $\eta^*>0$  such that $\eta^* \in I_k$ for all $k\geq k_0$. This will mean that $r_k=\sup I_k \geq \eta^*>0$ for all $k\geq k_0$.

 Let us show that $\left\lbrace y_k \right\rbrace_{k\geq 0}$ is a Cauchy sequence. To this end, we let $m\geq 0$ be arbitrary and $k\geq j>m$ be given. Then from Lemma  \ref{LE4.7}, for all $k\geq j>m$, we have
$$
\begin{array}{llll}
\Vert y_k-y_j \Vert &\leq &  \Vert y_k-T_{(A-\gamma B)_0}(l_{k}-l_{m})y_m \Vert
\\
\\
&+& \Vert T_{(A-\gamma B)_0}(l_{k}-l_{m})y_m-T_{(A-\gamma B)_0}(l_{j}-l_{m})y_m\Vert\\
\\
&+ & \Vert T_{(A-\gamma B)_0}(l_{j}-l_{m})y_m-y_j \Vert  \\
\\
&\leq & \Gamma \delta_\gamma(l_k-l_m)+M_\gamma e^{\omega_\gamma^+ (l_k-l_m)} (l_k-l_m) \\
\\
&+&  \Vert T_{(A-\gamma B)_0}(l_{k}-l_{m})y_m-T_{(A-\gamma B)_0}(l_{j}-l_{m})y_m\Vert \\
\\
&+&  \Gamma \delta_\gamma(l_j-l_m)+M_\gamma e^{\omega_\gamma^+ (l_j-l_m)} (l_j-l_m).
\end{array}
$$
Then
$$
\limsup_{k,j\rightarrow +\infty} \Vert y_k-y_j \Vert\leq 2 \Gamma \delta_\gamma(l^*-l_m)+2 M_\gamma e^{\omega_\gamma^+ (l^*-l_m)} (l^*-l_m).
$$
Since $m$ is arbitrary and
$$
\lim_{m\rightarrow +\infty} [2 \Gamma \delta_\gamma(l^*-l_m)+2 M_\gamma e^{\omega_\gamma^+ (l^*-l_m)} (l^*-l_m)]=0,
$$
we deduce that $(y_k)_{k\geq 0}$ is a Cauchy sequence in $B(0,\rho)\cap C_0$. Therefore there exists $y^*\in B(0,\rho)\cap C_0$ such that
$$
\lim_{k\rightarrow +\infty} y_k=y^* \in C_0.
$$
Since $y^* \in C_0$ we have 
$$
\lim_{h\rightarrow 0^+}\dfrac{1}{h} d\left( T_{(A-\gamma B)_0}(h)y^*+S_{A-\gamma B}(h)F_\gamma(l^*,y^*),C_0\right)=0.
$$
By using the above limit, we can find $\eta^*\in (0,\dfrac{\varepsilon}{4}) $ small enough such that 
\begin{equation}\label{4.16}
0<\eta^*< \dfrac{\varepsilon}{4}<\varepsilon^*
\end{equation}
and 
\begin{equation}\label{4.17}
\dfrac{1}{\eta^*} d\left( T_{(A-\gamma B)_0}(\eta^*)y^*+S_{A-\gamma B}(\eta^*)F_\gamma(l^*,y^*),C_0\right) \leq \dfrac{\varepsilon}{4}
\end{equation}
and  (by using the continuity of $(l,y) \to T_{(A-\gamma B)_0}(l)y$)
\begin{equation} \label{4.17Bis}
\Vert T_{(A-\gamma B)_0}(\eta^*)y^*-y^*\Vert \leq \dfrac{\varepsilon}{2}
\end{equation}
and (by using the continuity of $(l,y) \to F_\gamma(l,y)$)
\begin{equation}\label{4.18}
\vert l^*-l \vert \leq 2\eta^* \ \text{ and } \Vert y-y^* \Vert \leq 2\eta^* \Rightarrow \Vert F_\gamma(l,y)-F_\gamma(l^*,y^*) \Vert \leq \dfrac{\varepsilon}{2} \\.
\end{equation}
To obtain a contradiction, we will use the 1-Lipschitz continuity of $x\in X\rightarrow d(x,C_0)$ combined with the continuity of $(l,y) \to F_\gamma(l,y)$ and $(l,y) \to T_{(A-\gamma B)_0}(l)y$ at $(l^*,y^*)$. Thus there exists $k_0\geq 0$ large enough such that for all $k\geq k_0$ one has
\begin{equation}\label{4.19}
\left\lbrace
\begin{array}{llll}
\Vert F_\gamma(l_k,y_k)-F_\gamma(l^*,y^*) \Vert \leq \dfrac{\varepsilon}{2}\\
\Vert T_{(A-\gamma B)_0}(\eta^*)y_k-T_{(A-\gamma B)_0}(\eta^*)y^*\Vert\leq \dfrac{\varepsilon}{4}\\
\Vert y_k-y^* \Vert\leq \eta^* \text{ and }0<\vert l^*-l_k \vert \leq \eta^*
\end{array}
\right.
\end{equation}
since $\eta^*$ is fixed and $ y_k\to y^* $ and $l_k \to l^*$. % for each $k$ large enough we have $\Vert y_k-y^* \Vert\leq \eta^*$ and $0<\vert l^*-l_k \vert \leq \eta^*$.

By using \eqref{4.17} and $ y_k\to y^* $ and $l_k \to l^*$, we obtain for each $k\geq k_0$ (taking possibly $k_0$ larger) 
\begin{equation}\label{4.20}
\dfrac{1}{\eta^*} d\left( T_{(A-\gamma B)_0}(\eta^*)y_k+S_{A-\gamma B}(\eta^*)F_\gamma(l_k,y_k),C_0\right) < \dfrac{\varepsilon}{2},\ \forall k\geq k_0.
\end{equation}
Next we note that for any $k\geq k_0$
$$
0\leq  l-l_k \leq \eta^* \Rightarrow\vert l-l^* \vert\leq \vert l-l_k \vert+\vert l^*-l_k\vert\leq 2\eta^*
$$
and
$$
\Vert y-y_k \Vert \leq \eta^* \Rightarrow\Vert y-y^* \Vert\leq \Vert y-y_k \Vert+\Vert y^*-y_k\Vert\leq 2\eta^*.
$$
{ Combining \eqref{4.17}-\eqref{4.18} with \eqref{4.19}, it follows that} for any $k\geq k_0$
\begin{equation}\label{4.21}
\Vert F_\gamma(l,y)-F_\gamma(l_k,y_k) \Vert\leq \Vert F_\gamma(l,y)-F_\gamma(l^*,y^*) \Vert+\Vert F_\gamma(l^*,y^*)-F_\gamma(l_k,y_k) \Vert\leq \varepsilon
\end{equation}
whenever
\begin{equation}\label{4.22}
\vert l-l_k \vert \leq \eta^* \ \text{ and }\  \Vert y-y_k \Vert \leq \eta^*.
\end{equation}
{In view of \eqref{4.16}, \eqref{4.17} and \eqref{4.19}, we further have}
\begin{equation}\label{4.23}
\begin{array}{lll}
\Vert T_{(A-\gamma B)_0}(\eta^*)y_k-y_k \Vert &\leq & \Vert T_{(A-\gamma B)_0}(\eta^*)y_k-T_{(A-\gamma B)_0}(\eta^*)y^*\Vert \vspace{0.2cm}\\
& & +\Vert T_{(A-\gamma B)_0}(\eta^*)y^*-y^* \Vert+\Vert y^*-y_k \Vert\leq \varepsilon.
\end{array}
\end{equation}
Finally it follows from \eqref{4.20}-\eqref{4.23} that $0<\eta^* \in I_k$ for all $k\geq k_0$ which contradicts \eqref{4.15}. 
%$$
%0=\lim_{k\rightarrow +\infty} r_k=\lim_{k\rightarrow +\infty} \sup(I_k).
%$$
\end{proof}
\medskip

\noindent \textbf{Construction of the approximate solution:}
Recall that from property (i) of Lemma \ref{LE4.7} we have for each $m=0,\dots,k-1$ and each $k\geq 1$
\begin{equation}\label{4.24}
\begin{array}{lll}
 y_k&=& \displaystyle T_{(A-\gamma B)_0}(l_{k}-l_{m})y_m+\sum_{i=m}^{k-1} (l_{i+1}-l_{i})T_{(A-\gamma B)_0}(l_{k}-l_{i+1})H_i\\
 & & \displaystyle+\sum_{i=m}^{k-1} T_{(A-\gamma B)_0}(l_{k}-l_{i+1})S_{A-\gamma B}(l_{i+1}-l_i)F_\gamma(l_i,y_i).
\end{array}
\end{equation}
For each $t\in [l_k,l_{k+1}]$ and each $k=0,\dots,{n_\varepsilon}-1$, we set
\begin{equation}\label{4.25}
\begin{array}{lll}
u_\varepsilon(t)&:=& \displaystyle T_{(A-\gamma B)_0}(t-l_{0})y_0+S_{A-\gamma B}(t-l_k)F_\gamma(l_k,y_k)+(t-l_k)H_k\\
& & + \displaystyle \sum_{i=0}^{k-1} (l_{i+1}-l_{i})T_{(A-\gamma B)_0}(t-l_{i+1})H_i\\
 & & \displaystyle+\sum_{i=0}^{k-1} T_{(A-\gamma B)_0}(t-l_{i+1})S_{A-\gamma B}(l_{i+1}-l_i)F_\gamma(l_i,y_i)
\end{array}
\end{equation}
with the convention
$$
\sum_{i=m}^{p}=0\ \text{ if } p<m.
$$
By using the semigroup property for $t \to  T_{(A-\gamma B)_0}(t)$, we deduce from \eqref{4.24} and \eqref{4.25} that
\begin{equation} \label{4.26}
u_\varepsilon(t)=T_{(A-\gamma B)_0}(t-l_k)y_k+S_{A-\gamma B}(t-l_k)F_\gamma(l_k,y_k)+(t-l_k)H_k,\ \forall t\in [l_k,l_{k+1}].
\end{equation}
Then it is clear that $u_\varepsilon(t)$ is well defined and continuous from $[0,\tau]$ into $X_0$ and
$$
u_\varepsilon(l_k)=y_k,\ \forall k=0,\dots,n_\varepsilon.
$$
{Next we rewrite $u_\varepsilon(t)$ into a form that will be convenient for our subsequent discussions}. By using the relationship
$$
S_{A-\gamma B}(h+\sigma)-S_{A-\gamma B}(\sigma)=T_{(A-\gamma B)_0}(\sigma)S_{A-\gamma B}(h),\ \forall h\geq 0,\ \forall \sigma\geq 0
$$
one can rewrite from \eqref{4.25} the formula of $u_\varepsilon$ as
$$
\begin{array}{lll}
u_\varepsilon(t)&=& \displaystyle T_{(A-\gamma B)_0}(t-l_0)y_0+S_{A-\gamma B}(t-l_k)F_\gamma(l_k,y_k)+(t-l_k)H_k\\
& & \displaystyle  +\sum_{i=0}^{k-1} (l_{i+1}-l_i)T_{(A-\gamma B)_0}(t-l_{i+1})H_i\\ 
& & \displaystyle +\sum_{i=0}^{k-1}[S_{A-\gamma B}(t-l_{i+1})-S_{A-\gamma B}(t-l_i)]F_\gamma(l_i,y_i), \forall t\in [l_k,l_{k+1}].
\end{array}
$$
{Setting}
\begin{equation}\label{4.27}
f_\gamma(t)=F_\gamma(l_i,y_i),\ \forall t\in [l_i,l_{i+1}),\ i=0,\dots,n_{\varepsilon}-1,\ f_\gamma(l_{n_\varepsilon})=F_\gamma(l_{n_\varepsilon-1},y_{n_\varepsilon-1})
\end{equation}
and remembering that $y_0=x_0$, by Lemma \ref{LE4.2} we obtain for each $t\in [l_k,l_{k+1}]$,
\begin{equation}\label{4.28}
\begin{array}{lll}
u_\varepsilon(t)&=&\displaystyle T_{(A-\gamma B)_0}(t-l_0)x_0+(S_{A-\gamma B} \diamond f_\gamma(l_0+\cdot))(t-l_0)\\
& & \displaystyle +(t-l_k)H_k+\sum_{i=0}^{k-1} (l_{i+1}-l_i)T_{(A-\gamma B)_0}(t-l_{i+1})H_i.
\end{array}
\end{equation}
Similar arguments also gives for any $t\in [l_k,l_{k+1}]$ and each integer $m \in [0, k]$
\begin{equation}\label{4.29}
\begin{array}{lll}
u_\varepsilon(t)&=&\displaystyle T_{(A-\gamma B)_0}(t-l_m)y_m+(S_{A-\gamma B} \diamond f_\gamma(l_m+\cdot))(t-l_m)\\
& & \displaystyle +(t-l_k)H_k+\sum_{i=m}^{k-1} (l_{i+1}-l_i)T_{(A-\gamma B)_0}(t-l_{i+1})H_i.
\end{array}
\end{equation}
By using again \eqref{4.2}, we also have the following estimate that for any $t\in [l_m,l_{k}]$ with $k\geq m$,
\begin{equation}\label{4.30}
\Vert (S_{A-\gamma B} \diamond f_\gamma(l_m+\cdot))(t-l_m) \Vert \leq  \Gamma\delta_\gamma (t-l_m) .
\end{equation}
\begin{lemma}\label{LE4.9}
Let Assumptions \ref{ASS2.1}, \ref{ASS2.4},  \ref{ASS3.1} and \ref{ASS4.4}  be satisfied. Then the approximate solution $u_\varepsilon(t)$ in \eqref{4.28} satisfies the following properties
\begin{itemize}
\item[{\rm (i)}] There exits a constant  $\hat{M}_0>0$ such that
$$
\Vert u_\varepsilon(t)-y_k \Vert \leq \hat{M}_0(\varepsilon+\delta_\gamma(\varepsilon)),\ \forall t\in [l_k,l_{k+1}]
$$
with $k=0,\dots,n_\varepsilon-1$.
\item[{\rm (ii)}] $u_\varepsilon(t) \in B(0,\rho),\ \forall t\in [0,\tau]$.
\item[{\rm (iii)}] There exists a {constant} $\hat{M}_1>0$  such that for all $t\in [0,\tau]$
\begin{equation}\label{4.31}
\left\Vert u_\varepsilon(t)-T_{(A-\gamma B)_0}(t)x_0-(S_{A-\gamma B} \diamond F_\gamma(\cdot,u_\varepsilon(\cdot))(t)\right\Vert\leq  \hat{M}_1(\varepsilon+\delta_\gamma(\varepsilon)).
\end{equation}
\end{itemize}
\end{lemma}
\begin{proof}
{We first prove that, for each $t\in [l_m,l_p]$ with $p\geq m\geq 0$ and each $\bar{y}\in X_0$, we have}
\begin{equation}\label{4.32}
\Vert u_\varepsilon(t)-\bar{y}\Vert \leq \Vert T_{(A-\gamma B)_0}(t-l_m)y_m-\bar{y}\Vert +\Gamma \delta_\gamma(t-l_m)+\dfrac{\varepsilon}{2} M_\gamma(t-l_m) e^{\omega_\gamma^+(t-l_m)}.
\end{equation}
Let $p> m\geq 0$ be given. From \eqref{4.29} we have
\begin{equation*}
\begin{array}{lll}
u_\varepsilon(t)-\bar{y} &=&\displaystyle T_{(A-\gamma B)_0}(t-l_m)y_m-\bar{y}+(S_{A-\gamma B} \diamond f_\gamma(l_m+\cdot))(t-l_m)\\
& & \displaystyle +(t-l_k)H_k+\sum_{i=m}^{k-1} (l_{i+1}-l_i)T_{(A-\gamma B)_0}(t-l_{i+1})H_i,\ \forall t\in [l_k,l_{k+1}]
\end{array}
\end{equation*}
with $m\leq k\leq p-1$.
Hence
\begin{equation*}
\begin{array}{lll}
\Vert u_\varepsilon(t)-\bar{y} \Vert &\leq &\displaystyle \Vert T_{(A-\gamma B)_0}(t-l_m)y_m-\bar{y}\Vert +\Vert (S_{A-\gamma B} \diamond f_\gamma(l_m+\cdot))(t-l_m)\Vert \\
& & \displaystyle +(t-l_k)\Vert H_k\Vert +\sum_{i=m}^{k-1} (l_{i+1}-l_i)\Vert T_{(A-\gamma B)_0}(t-l_{i+1})H_i\Vert,\ \forall t\in [l_k,l_{k+1}]
\end{array}
\end{equation*}
with $m\leq k\leq p-1$. {In view of \eqref{4.30}, and}
$$
H_i \in X_0\ \text{ and } \Vert H_i \Vert\leq \dfrac{\varepsilon}{2},\  i=0,\dots,n_\varepsilon,
$$
{we see that, for any $t\in [l_k,l_{k+1}]$ with $m\leq k\leq p-1$,}
\begin{equation*}
\begin{array}{lll}
\Vert u_\varepsilon(t)-\bar{y} \Vert &\leq &\displaystyle \Vert T_{(A-\gamma B)_0}(t-l_m)y_m-\bar{y}\Vert +\Gamma \delta_\gamma(t-l_m) +(t-l_k) \dfrac{\varepsilon}{2} \\
& & \displaystyle +\sum_{i=m}^{k-1}M_\gamma e^{\omega_\gamma (t-l_{i+1})} \dfrac{\varepsilon}{2} (l_{i+1}-l_i)\\
&\leq &\displaystyle \Vert T_{(A-\gamma B)_0}(t-l_m)y_m-\bar{y}\Vert + \Gamma \delta_\gamma(t-l_m)+\dfrac{\varepsilon}{2} M_\gamma(t-l_m) e^{\omega_\gamma^+(t-l_m)}\\
&\leq &\displaystyle \Vert T_{(A-\gamma B)_0}(t-l_m)y_m-\bar{y}\Vert +\Gamma \delta_\gamma(t-l_m)+\dfrac{\varepsilon}{2} M_\gamma(t-l_m) e^{\omega_\gamma^+(t-l_m)},
\end{array}
\end{equation*}
which {proves} \eqref{4.32}.

\noindent \textbf{Proof of (i):} By using \eqref{4.32} with $m=k$, $p=k+1$ and $\bar{y}=y_k${, for each $t\in [l_k,l_{k+1}]$, it follows that}
$$
\Vert u_\varepsilon(t)-y_k \Vert \leq \Vert T_{(A-\gamma B)_0}(t-l_k)y_k-y_k\Vert +\Gamma \delta_\gamma(t-l_k)+\dfrac{\varepsilon}{2} M_\gamma(t-l_k) e^{\omega_\gamma^+(t-l_k)}.
$$
{Observing that}
$$
t\in [l_k,l_{k+1}] \Rightarrow t-l_k \leq l_{k+1}-l_k\leq \dfrac{r_k}{2}<r_k\leq \varepsilon  \Rightarrow t-l_k \in I_k
$$
where $I_k$ and $r_k$ are defined respectively in \eqref{4.7}  and \eqref{4.8}. {Then} we deduce that
$$
\Vert T_{(A-\gamma B)_0}(t-l_k)y_k-y_k\Vert\leq \varepsilon,\ \forall t\in [l_k,l_{k+1}]
$$
and
\begin{equation}\label{4.33}
\Vert u_\varepsilon(t)-y_k \Vert \leq \varepsilon +\Gamma \delta_\gamma(\varepsilon)+\dfrac{\varepsilon}{2} M_\gamma \varepsilon e^{\omega_\gamma^+\varepsilon },\ \forall t\in [l_k,l_{k+1}].
\end{equation}
This proves (i). \\
\noindent \textbf{Proof of (ii):} {In view of \eqref{4.32} with $m=0$, $p=n_\varepsilon$ and $\bar{y}=0$, and using the fact $l_0=0$ and $y_0=x_0$, we deduce that}
$$
\begin{array}{llll}
\Vert u_\varepsilon(t)\Vert \leq   \Vert T_{(A-\gamma B)_0}(t)x_0\Vert+\Gamma \delta_\gamma(t)+M_\gamma e^{\omega_\gamma^+ t}t,\ \forall t\in [0,\tau].
\end{array}
$$
{Then the fact $0\leq t \leq \tau$ together with the inequality \eqref{4.5} imply that}
$$
\Vert u_\varepsilon(t) \Vert\leq \rho, \ \forall t\in [0,\tau].
$$
\noindent  \textbf{Proof of (iii):} Let
\begin{equation*}
v _\varepsilon(t)=u_\varepsilon(t)-T_{(A-\gamma B)_0}(t)x_0-(S_{A-\gamma B} \diamond F_\gamma(\cdot,u_\varepsilon(\cdot))(t),\ \forall t\in [0,\tau].
\end{equation*}
{We further define}
$$
g_\gamma(t):=f_\gamma(t)- F_\gamma(t,u_\varepsilon(t)),\ \forall t\in [0,\tau]
$$
or equivalently 
\begin{equation}\label{4.34}
g_\gamma(t)=\left\lbrace\begin{array}{lll}
 F_\gamma(l_k,y_k)-F_\gamma(t,u_\varepsilon(t)) & \text{ if } &  t\in [l_k,l_{k+1}),\ k=0,\dots,n_\varepsilon-1\\
F_\gamma(l_{n_\varepsilon-1},y_{n_\varepsilon-1})-F_\gamma(l_{n_\varepsilon},y_{n_\varepsilon}) & \text{ if } & t=l_{n_\varepsilon}.
\end{array}
\right.
\end{equation}
where $f_\gamma$ is defined in \eqref{4.27} and $n_\varepsilon$ has been defined in Lemma \ref{LE4.8}. 

Then using \eqref{4.29} we get
$$
v_\varepsilon(t)=(S_{A-\gamma B} \diamond g_\gamma(\cdot))(t)+(t-l_k)H_k+\sum_{i=0}^{k-1} (l_{i+1}-l_i)T_{(A-\gamma B)_0}(t-l_{i+1})H_i,\ \forall t\in [l_k,l_{k+1}].
$$
Since $g_\gamma \in {\rm Reg}([0,\tau],X)$, it follows that
$$
\begin{array}{lll}
\Vert v_\varepsilon(t)\Vert &\leq & \displaystyle\delta_\gamma(t) \sup_{s \in [0,t]} \Vert g_\gamma(s) \Vert+\dfrac{\varepsilon}{2}M_\gamma(t-l_0) e^{\omega_\gamma^+t}\\
&\leq &\displaystyle \delta_\gamma(\tau) \sup_{s \in [0,t]} \Vert g_\gamma(s) \Vert+\dfrac{\varepsilon}{2}M_\gamma\tau  e^{\omega_\gamma^+\tau}.
\end{array}
$$
Therefore one can obtain \eqref{4.31} by estimating
$$
\sup_{s \in [0,t]} \Vert g_\gamma(s) \Vert,\ \forall t\in [0,\tau].
$$
In view of \eqref{4.34}, it follows that
$$
\Vert g_\gamma(t) \Vert\leq  \Vert   F_\gamma(l_k,y_k)-F_\gamma(t,y_k) \Vert+\Vert   F_\gamma(t,y_k)-F_\gamma(t,u_\varepsilon(t)) \Vert,\ t\in [l_k,l_{k+1}]
$$
with $k=0,\dots,n_\varepsilon$. {Observing that if $t\in [l_k,l_{k+1}]$, then}
$$
t-l_k\leq l_{k+1}-l_k\leq \dfrac{r_k}{2}< r_k\leq \rho  \Rightarrow t-l_k \in I_k \text{ and } t\in [0,\rho]
$$
where $I_k$ and $r_k$ are defined respectively in \eqref{4.7}  and \eqref{4.8}. {This observation together with the fact}
$$
u_\varepsilon(t) \in B(0,\rho), \ \forall t\in [0,\tau],
$$
{implying that}
$$
\Vert g_\gamma(t) \Vert \leq  \varepsilon+\Lambda \Vert y_k-u_\varepsilon(t) \Vert,\ \forall t\in [l_k,l_{k+1}].
$$
Finally we infer from \eqref{4.33}  that
$$
\Vert g_\gamma(t) \Vert \leq  \varepsilon+\Lambda [\varepsilon +\Gamma \delta_\gamma(\varepsilon)+\dfrac{\varepsilon}{2} M_\gamma \varepsilon e^{\omega_\gamma^+\varepsilon }],\ \forall t\in [l_k,l_{k+1}].
$$
The result follows.
\end{proof}\\

\noindent \textbf{Existence of solution in $C_0$:} At this stage, the approximated solution $t \to u_\varepsilon(t)$ only belongs to $C_0$ for $t=l_k$ (since $u(l_k)=y_k \in C_0$). In this last part of the proof, we take the limit when $\varepsilon \to 0$ and after proving that the limit exits (by using Cauchy sequences), we will prove that the limit solution takes his value in $C_0$.

We first prove that the approximated solution $(u_\varepsilon)_{\varepsilon \in (0,\varepsilon^*)}$ forms a Cauchy sequence in $C([0,\tau],X_0)$ and its limit is a solution of system  \eqref{1.1}. Indeed, by using property (iii) of Lemma \ref{LE4.9}, we have
$$
\Vert u_\varepsilon(t)-u_\sigma (t)\Vert\leq  \hat{M}_1 [\varepsilon+\delta_\gamma(\varepsilon)+\sigma+\delta_\gamma(\sigma)]+\delta_\gamma(t) \sup_{s\in [0,t]}\Vert F_\gamma(s,u_\varepsilon(s))-F_\gamma(s,u_\sigma(s)) \Vert.
$$
{Since}
$$
u_\varepsilon(t),  u_\sigma(t)\in B(0,\rho),\ \forall t\in [0,\tau],\ 0<\tau\leq \rho,
$$
we obtain
$$
\Vert u_\varepsilon(t)-u_\sigma (t)\Vert\leq  \hat{M}_1 [\varepsilon+\delta_\gamma(\varepsilon)+\sigma+\delta_\gamma(\sigma)]+\delta_\gamma(\tau) \Lambda \sup_{s\in [0,\tau]} \Vert  u_\varepsilon(s)-u_\sigma(s) \Vert,\ \forall t\in [0,\tau].
$$
{In view of \eqref{4.6}, we have $0<\delta_\gamma(\tau) \Lambda<1$, and hence,}
$$
\sup_{t\in [0,\tau]} \Vert u_\varepsilon(t)-u_\sigma (t)\Vert\leq \dfrac{\hat{M}_1}{1-\delta_\gamma(\tau) \Lambda}  [\varepsilon+\delta_\gamma(\varepsilon)+\sigma+\delta_\gamma(\sigma)].
$$
Therefore $(u_\varepsilon)_{\varepsilon \in (0,\varepsilon^*)} \in C([0,\tau],X_0)$ is a Cauchy sequence in $C([0,\tau],X_0)$ endowed with {the supremum norm. Then} there exists $u \in C([0,\tau],X_0)$  such that
$$
\lim_{\varepsilon \rightarrow 0^+} \sup_{t\in [0,\tau]} \Vert u_\varepsilon(t)-u(t) \Vert=0.
$$
{Letting $\varepsilon$ tend to zero in \eqref{4.31}, it} is straightforward that
$$
\begin{array}{lll}
u(t)&=& T_{(A-\gamma B)_0}(t)x_0+(S_{A-\gamma B} \diamond F_\gamma(\cdot,u(\cdot))(t)\\
&=&T_{A_0}(t)x_0+(S_{A} \diamond F(\cdot,u(\cdot))(t),\ \forall t\in [0,\tau],\ \forall t\in [0,\tau].
\end{array}
$$
{That} is to say that $u \in C([0,\tau],X_0)$ is a mild solution of \eqref{1.1} in $[0,\tau]$. Finally using property (i) of Lemma \ref{LE4.9}, {we see that}
$$
d(u_\varepsilon(t),C_0) \leq \hat{M}_0 (\varepsilon+\delta_\gamma(\varepsilon)),\ \forall t\in [0,\tau] \Rightarrow\lim_{\varepsilon\rightarrow 0^+} d(u_\varepsilon(t),C_0) =0,\ \forall t\in [0,\tau].
$$
{By the continuity of  $x\in X_0 \mapsto d(x,C_0)$, we further see that}
$$
d(u(t),C_0) =\lim_{\varepsilon\rightarrow 0^+} d(u_\varepsilon(t),C_0) ,\ \forall t\in [0,\tau] \Rightarrow u(t)\in C_0,\ \forall t\in [0,\tau].
$$
\section{Applications to age structured models}
We will consider a generalization of the one dimensional model presented in \cite{Thieme90a}. The model considered is the following
\begin{equation}\label{5.1}
\left\{ 
\begin{array}{lll}
\displaystyle \dfrac{\partial u(t,a)}{\partial t}+\dfrac{\partial u(t,a)}{\partial a}=-\mu(a) u(t,a)\left(\kappa-\Theta(u(t,a))\right) \\
\displaystyle  u(t,0)=\int_0^{+\infty} \beta(a) u(t,a) \left(\kappa-\Theta(u(t,a))\right) da \\
u(0,.)=u_0\in \Li^p_+(\mathbb{R}_+,\mathbb{R}^n),\ p\in [1,+\infty)
\end{array}
\right.
\end{equation}
where we have set
\begin{equation*}
\Theta(x)=\sum_{k=0}^n x_k,\ \forall x \in \mathbb{R}^n
\end{equation*}
and assume that $\kappa>0$,  $\beta,\mu \in  \Li^{\infty}_+(\mathbb{R}_+,\mathbb{R})$ with $\frac{1}{p}+\frac{1}{q}=1$ and
$$
\beta(a)=0,\ \forall a \geq a_\dagger \ \text{ and } \ \mu(a)\geq \mu_->0,\ \forall a \geq 0.
$$
It is important to note that the model \eqref{5.1} is not well defined in  $\Li^p(\mathbb{R}_+,\mathbb{R}^n)$ however it does in a proper subset of $\Li^p(\mathbb{R}_+,\mathbb{R}^n)$ namely 
\begin{equation} \label{5.2}
C=\left \{ \varphi \in  \Li^p_+(\mathbb{R}_+,\mathbb{R}^n) : 0\leq \Theta(\varphi(a)) \leq \kappa \text{ for a.e.} \ a\geq 0 \right \}.
\end{equation}
\noindent \textbf{Truncated system:} The interest of our result is that we will be able to demonstrate the existence of solutions for initial data in $C$. To do so we introduce the following truncation function $\chi : \mathbb{R}\rightarrow [0,\kappa]$ defined by 
$$
\chi(s)=\min(k,s^+),\ \forall s \in \mathbb{R}
$$
and we set for each $i=1,\dots,n$
\begin{equation}\label{5.3}
\left\{ 
\begin{array}{lll}
\displaystyle \dfrac{\partial u_i(t,a)}{\partial t}+\dfrac{\partial u_i(t,a)}{\partial a}=-\mu(a) \chi(u_i(t,a))\chi\left(\kappa- \Theta(u(t,a))\right) \\
\displaystyle  u_i(t,0)=\int_0^{+\infty} \beta(a) \chi(u_i(t,a))\chi\left(\kappa- \Theta(u(t,a))\right) da \\
u_i(0,.)=u_{i0}\in \Li^p_+(\mathbb{R}_+,\mathbb{R}),\ p\in [1,+\infty)
\end{array}
\right.
\end{equation}
which is well defined in $\Li^p_+(\mathbb{R}_+,\mathbb{R}^n)$. The idea is to prove that for each $\varphi \in C$ there exists a unique mild solution of \eqref{5.3} lying in $C$ and since the two systems coincide in $C$ the result follows.\\
\textbf{Abstract reformulation: } Set
$$
X=\mathbb{R}^n \times \Li^p(\mathbb{R}_+,\mathbb{R}^n)
$$
endowed with the usual product norm. Consider the linear operator $A:D(A)\subset X \to X$
$$
A
\left(
\begin{array}{c}
0_{\mathbb{R}^n}\\
\varphi
\end{array}
\right)
=
\left(
\begin{array}{c}
-\varphi(0)\\
-\varphi^\prime
\end{array}
\right)
$$
and
$$
D(A)=\left\{ 0_{\mathbb{R}^n} \right\} \times \W^{1,p}(\mathbb{R}_+, \mathbb{R}^n)
$$
and note that the closure of the domain of $A$ is
$$
X_0:=\overline{D(A)}=\left\{ 0_{\mathbb{R}^n} \right\} \times \Li^p(\mathbb{R}_+,\mathbb{R}^n).
$$
Consider the non linear maps $F_0 :\Li^p(\mathbb{R}_+,\mathbb{R}^n) \rightarrow \mathbb{R}^n$ and  $F_1 :\Li^p(\mathbb{R}_+,\mathbb{R}^n)\to \Li^p(\mathbb{R}_+,\mathbb{R}^n)$  defined respectively for each $i=1,\dots,n$ by
$$
F_0(\varphi)_i= \int_0^{+\infty} \beta(a) \chi(\varphi_i(a)) \chi\left(\kappa-\Theta(\varphi(a))\right) da,\ \text{ for a.e } a\geq 0
$$
and 
$$
F_1(\varphi)_i(a)=-\mu(a)\chi(\varphi_i(a)) \chi\left(\kappa-\Theta(\varphi(a))\right),\ \text{ for a.e } a\geq 0.
$$
Next we consider $F : X_0 \rightarrow X$ defined by 
$$
F
\left(
\begin{array}{c}
0_{\mathbb{R}^n}\\
\varphi
\end{array}
\right)
=\left(
\begin{array}{c}
F_0(\varphi)\\
F_1(\varphi)
\end{array}
\right).
$$
By identifying $u(t,.)$ with $v(t):=
\left(
\begin{array}{c}
0_{\mathbb{R}^n}\\
u(t,.)
\end{array}
\right)$ we can rewrite the partial differential equation \eqref{5.1} as the following abstract Cauchy problem
\begin{equation*}
v^\prime(t)=Av(t)+F(v(t)), \text{ for } t \geq 0, \ v(0)=\left(
\begin{array}{c}
0_{\mathbb{R}^n}\\
u_0
\end{array}
\right) \in X_0.
\end{equation*}
It is well known that the linear operator $A : D(A)\subset X_0\rightarrow X_0$ is not Hille-Yosida for $p>1$ but fulfill the conditions of Assumption \ref{ASS2.1} (see \cite[Section 6]{Magal-Ruan07}). By using similar arguments in \cite{Magal-Ruan07} one can also show that Assumption \ref{ASS2.4} is satisfied. It can be easily checked that $F$ is Lipschitz on bounded sets of $X_0$. Therefore in what follow we will only verify that Assumption \ref{ASS4.4} is satisfied.  We consider the following closed subset as a candidate for the application of our results 
\begin{equation} \label{5.4}
\mathcal{C}_0=\{ 0_{\mathbb{R}^n}\} \times C.
\end{equation}
In order to verify Assumption \ref{ASS4.4} we will first determine the strongly continuous semigroup $\{T_{A_0}(t) \}_{t\geq 0} \subset \mathcal{L}(X_0)$  generated by $A_0$ the part of $A$ in $X_0$ and the integrated semigroup  $\{S_{A}(t) \}_{t\geq 0} \subset \mathcal{L}(X)$ generated by $A$. Indeed $\{T_{A_0}(t) \}_{t\geq 0} \subset \mathcal{L}(X_0)$ is given by 
\begin{equation*}
T_{A_0}(t)\left(
\begin{array}{c}
0_{\mathbb{R}^n}\\
\varphi
\end{array}
\right)=\left(
\begin{array}{c}
0_{\mathbb{R}^n}\\
\widehat{T}_{A_0}(t)(\varphi)
\end{array}
\right),\ \forall \left( \begin{array}{cc}
0_{\mathbb{R}^n}\\
\varphi 
\end{array}
\right) \in X_0
\end{equation*}
with $t\rightarrow  \widehat{T}_{A_0}(t)(\varphi)$ the unique continuous mild solution of the partial differential equation 
\begin{equation*}
\left\{ 
\begin{array}{lll}
\dfrac{\partial u(t,a)}{\partial t}+\dfrac{\partial u(t,a)}{\partial a}=0,\ t>0,\ a>0 \\
\displaystyle  u(t,0)=0,\ t>0 \\
u(0,.)=\varphi \in \Li^p(\mathbb{R}_+,\mathbb{R}^n).
\end{array}
\right.
\end{equation*}
Thus integrating along the characteristics yields
\begin{equation}\label{5.5.0}
\widehat{T}_{A_0}(t)(\varphi)(a)=
\left\lbrace
\begin{array}{l}
\varphi (a-t), \text{ if } a>t,\\
0, \text{ if } a<t,\\
\end{array}
\right.
\end{equation}
which can be rewritten into the more condensed form 
\begin{equation}\label{5.5}
\widehat{T}_{A_0}(t)(\varphi)(a)=H(a-t)\varphi(a-t),\ \forall t\geq 0,\ \text{ for a.e. } a \geq 0
\end{equation}
where the map $\varphi(a)$ is understood as its extension by $0$ for almost every $a < 0$ and $a\rightarrow H(a)$ is the Heaviside function defined by 
$$
H(a)=1 \ \text{ if } a \geq 0 \ \text{ and } H(a)=0,\ \text{ if } a<0.
$$
Furthermore the integrated semigroup generated by $A$ is given by 
\begin{equation*}
S_{A}(t)\left(
\begin{array}{c}
x\\
\varphi
\end{array}
\right)=\left(
\begin{array}{c}
0_{\mathbb{R}^n}\\
\widehat{S}_{A}(t)(x,\varphi)
\end{array}
\right),\ \forall \left( \begin{array}{cc}
x\\
\varphi 
\end{array}
\right) \in X
\end{equation*}
with $t\rightarrow \widehat{S}_A(t)(x,\varphi)$ the unique mild solution of the partial differential equation 
\begin{equation*}
\left\{ 
\begin{array}{lll}
\dfrac{\partial u(t,a)}{\partial t}+\dfrac{\partial u(t,a)}{\partial a}=\varphi(a),\ t>0,\ a>0, \\
\displaystyle  u(t,0)=x,\ t>0, \\
u(0,.)=0_{\Li^p},
\end{array}
\right.
\end{equation*}
which is obtained by integrating along the characteristics as follow
\begin{equation}
\widehat{S}_A(t)(x,0)(a)=
\left\lbrace
\begin{array}{l}
x, \text{ if } t>a,\\
0, \text{ if } t<a,
\end{array} 
\right.
\end{equation}
and 
\begin{equation*}
\widehat{S}_A(t)(0,\varphi)(a)=\left( \int_0^t \widehat{T}_{A_0}(l)\left(\varphi\right) dl \right) \left(a \right)
\end{equation*}
therefore (since by linearity $\widehat{S}_A(t)(x,\varphi)=\widehat{S}_A(t)(x,0)+\widehat{S}_A(t)(0,\varphi)$ we obtain 
\begin{equation}\label{5.6}
\widehat{S}_A(t)(x,\varphi)(a)=(1- H(a-t))x+\left( \int_0^t \widehat{T}_{A_0}(l)\left(\varphi\right) dl \right) \left(a \right),\ \forall t\geq 0,\ \text{ for a.e. } a \geq 0.
\end{equation}
The following lemma will allows us to give a more explicit form of \eqref{5.6}.
\begin{lemma}\label{LE5.1.0}
For each $t\geq 0$ we have
$$
\widehat{S}_A(t)(0,\varphi)(a)=\int_0^t \widehat{T}_{A_0}(l)\left(\varphi\right)(a)dl=\int_0^t H(a-l) \varphi(a-l) dl,\ \text{ for a.e. } a\geq 0
$$
where $H$ is the Heaviside function. Moreover we have 

\begin{equation}\label{5.9.0}
\widehat{S}_A(t)(x,\varphi)(a)=(1- H(a-t))x+\int_0^t H(a-l) \varphi(a-l) dl,\ \forall t\geq 0,\ \text{ a.e. } a \geq 0.
\end{equation}
\end{lemma} 
\begin{proof}
Let $x^* : L^p(\mathbb{R}_+,\mathbb{R}^n)\rightarrow \mathbb{R}$ any linear continuous functional. Then by the Riesz representation theorem there exists a unique $\psi \in  L^q(\mathbb{R}_+,\mathbb{R}^n)$ with $\frac{1}{p}+\frac{1}{q}=1$ such that 
$$
x^*(\phi)=\int_0^{+\infty}\psi(a) \phi(a)da,\ \forall \phi \in L^p(\mathbb{R}_+,\mathbb{R}^n).
$$
Therefore we have by using Fubini's theorem for each $t\geq 0$
$$
\begin{array}{llll}
x^*\left( \widehat{S}_A(t)(0,\varphi) \right)&=&\displaystyle  x^* \left(\int_0^t \widehat{T}_{A_0}(l)\left(\varphi\right) dl\right) \\
&=& \displaystyle \int_0^t  x^*\left(\widehat{T}_{A_0}(l)\left(\varphi\right)\right) dl \\
&=& \displaystyle  \int_0^t \int_0^{+\infty} \psi(a)\widehat{T}_{A_0}(l)(\varphi)(a)da dl\\
&=& \displaystyle  \int_0^t \int_l^{+\infty} \psi(a)\varphi(a-l)da dl\\
&=& \displaystyle  \int_0^t \int_0^{a} \psi(a)\varphi(a-l)dl da+ \int_t^\infty \int_0^{t} \psi(a)\varphi(a-l)dl da  \\
&=& \displaystyle  \int_0^\infty \int_0^{\min(a,t)} \psi(a)\varphi(a-l)dl da  \\
&=& \displaystyle  \int_0^\infty \psi(a) \int_0^{t} H(a-l)\varphi(a-l)dl da. \\
\end{array}
$$
Since $x^*$ is arbitrary, by using the Hahn-Banach theorem we deduce that 
$$
 \widehat{S}_A(t)(0,\varphi)(a)=\int_0^t \widehat{T}_{A_0}(l)(\varphi)(a) dl,\ \forall t\geq 0,\ \text{ for a.e. } a \geq 0
$$
and the result follows by using \eqref{5.5.0}.
\end{proof}\\

Hence using \eqref{5.5} and \eqref{5.6} we have for each $\left(
\begin{array}{c}
0_{\mathbb{R}^n}\\
\varphi
\end{array}
\right)\in \mathcal{C}_0$ 
\begin{equation*}
T_{A_0}(h)\left(
\begin{array}{c}
0_{\mathbb{R}^n}\\
\varphi
\end{array}
\right)+S_{A}(h)F\left(
\begin{array}{c}
0_{\mathbb{R}^n}\\
\varphi
\end{array}
\right)=\left(
\begin{array}{c}
0_{\mathbb{R}^n}\\
\widehat{v}(\varphi;h)
\end{array}
\right),\ \forall h\geq 0
\end{equation*}
with 
$$
\widehat{v}(\varphi;h)=\widehat{T}_{A_0}(h)(\varphi)+\widehat{S}_A(h)(F_0(\varphi),F_1(\varphi)),\ \forall h\geq 0.
$$
More precisely by using \eqref{5.5} and  \eqref{5.9.0} we have
\begin{equation}\label{5.7}
\widehat{v}(\varphi;h)(a)=H(a-h)\varphi(a-h)+[1-H(a-h)]F_0(\varphi)+\int_0^h H(a-l) F_1(\varphi)(a-l) dl,\ \forall h\geq 0,\ \forall a \geq 0
\end{equation}
or equivalently 
\begin{equation}\label{5.8}
\widehat{v}(\varphi;h)(a)=\widehat{v}_1(\varphi;h)(a)+\widehat{v}_2(\varphi;h)(a),\ \forall h\geq 0,\ \forall a \geq 0
\end{equation}
with 
\begin{equation}\label{5.9}
\left\lbrace
\begin{array}{lll}
\widehat{v}_1(\varphi;h)(a)=&H(a-h)\varphi(a-h)+[1-H(a-h)]F_0(\varphi)\\
\vspace{-0.3cm} \\
&+h H(a-h)F_1(\varphi)(a-h)\\
\vspace{-0.1cm} \\
\displaystyle \widehat{v}_2(\varphi;h)(a)=& h [F_1(\varphi)(a)-H(a-h)F_1(\varphi)(a-h)]\\
\vspace{-0.3cm} \\
&+\int_0^h  [H(a-l)F_1(\varphi)(a-l)-F_1(\varphi)(a)]dl.
\end{array}
\right.
\end{equation}
\begin{lemma}\label{LE5.1}
For each $\varphi \in C$ we have 
$$
\lim_{h\rightarrow 0^+}\dfrac{1}{h} \Vert \widehat{v}_2(\varphi;h) \Vert_{L^p}=0.
$$
\end{lemma}
\begin{proof} We will give the proof for $1<p<+\infty$. The case $p=1$ can be obtained easily.  Let $q\in (1,+\infty)$ be given such that $\frac{1}{p}+\frac{1}{q}=1$. We have for each $h>0$
$$
\begin{array}{lllll}
\Vert \widehat{v}_2(\varphi;h) \Vert_{L^p}&\leq & h\Vert H(.) F_1(\varphi)(.)-H(.-h)F_1(\varphi)(.-h) \Vert_{L^p} \\
& & \displaystyle +\left(\int_0^{+\infty}\left(\int_0^h \Vert H(a-l)F_1(\varphi)(a-l)-F_1(\varphi)(a)\Vert dl \right)^p da\right)^{\frac{1}{p}}\\
&\leq & h \Vert H(.) F_1(\varphi)(.)-H(.-h)F_1(\varphi)(.-h) \Vert_{L^p} \\
& & \displaystyle +h^\frac{1}{q}\left(\int_0^{+\infty} \int_0^h  \Vert H(a-l)F_1(\varphi)(a-l)-F_1(\varphi)(a)\Vert^p dl da \right)^{\frac{1}{p}}\\
&\leq & h \Vert H(.) F_1(\varphi)(.)-H(.-h)F_1(\varphi)(.-h) \Vert_{L^p} \\
& & \displaystyle +h^\frac{1}{q}\left( \int_0^h \Vert H(.-l)F_1(\varphi)(.-l) -H(.) F_1(\varphi)(.)\Vert_{L^p}^p dl \right)^{\frac{1}{p}}\\
&\leq & h \Vert H(.) F_1(\varphi)(.)-H(.-h)F_1(\varphi)(.-h) \Vert_{L^p} \\
& & \displaystyle +h\left( \int_0^1 \Vert H(.-lh)F_1(\varphi)(.-lh) -H(.) F_1(\varphi)(.)\Vert_{L^p}^p dl \right)^{\frac{1}{p}}
\end{array}
$$
and the result follows by using the continuity of the translation in $L^p$.
\end{proof}\\

\noindent Note that since the $d(.;\mathcal{C}_0)$ is $1$-Lipschitz continuous we have 
$$
0\leq\dfrac{1}{h} d\left(\left(\begin{array}{c}
0_{\mathbb{R}^n}\\
\widehat{v}(\varphi;h)
\end{array} \right);\mathcal{C}_0\right)\leq \dfrac{1}{h}  d\left(\left(\begin{array}{c}
0_{\mathbb{R}^n}\\
\widehat{v}_1(\varphi;h)
\end{array} \right);\mathcal{C}_0\right)+\dfrac{1}{h}\Vert \widehat{v}_2(\varphi;h) \Vert_{L^p},\ \forall h>0
$$
it now follows from Lemma \ref{LE5.1} that Assumption \ref{ASS4.4} is satisfied if 
\begin{equation}\label{5.10}
\lim_{h\rightarrow 0^+}\dfrac{1}{h}  d\left(\left(\begin{array}{c}
0_{\mathbb{R}^n}\\
\widehat{v}_1(\varphi;h)
\end{array} \right);\mathcal{C}_0\right)=0.
\end{equation}
In order to prove \eqref{5.10} we will show that under some conditions to be make precise later $\left(\begin{array}{c}
0_{\mathbb{R}^n}\\
\widehat{v}_1(\varphi;h)
\end{array} \right)$ belongs to $\mathcal{C}_0$ for $h>0$ sufficiently small. To this end note that 
\begin{equation}\label{5.11}
\widehat{v}_1(\varphi;h)(a)=\left\lbrace
\begin{array}{llll}
\varphi(a-h)+h F_1(\varphi)(a-h) & \text{ if } & a\geq h \\
F_0(\varphi) & \text{ if } & a< h.
\end{array}
\right.
\end{equation}
\begin{lemma} \label{LE5.2}
Assume that 
\begin{equation}\label{5.12}
\int_0^{a_\dagger} \beta(a) da \leq \dfrac{4}{\kappa}
\end{equation}
Then there exists $h_0>0$ such that for each $\varphi \in C$  we have 
$$
\widehat{v}_1(\varphi;h)\in C,\  \forall h\in (0,h_0).
$$
\end{lemma}
\begin{proof}
Let  $\varphi \in C$  be given. Since $\Theta$ is linear, if $a\geq h$ then by \eqref{5.11} we have
$$
\begin{array}{llll}
\Theta\left( \widehat{v}_1(\varphi;h)(a)\right)&=&\Theta\left( \varphi(a-h)-h \mu(a-h) \varphi(a-h) \left[\kappa-\Theta(\varphi(a-h))\right]\right)\\
&=& (1-h\mu(a-h) [\kappa-\Theta(\varphi(a-h)])\Theta(\varphi(a-h))\\
\end{array}
$$
hence 
$$
[1-h\Vert\mu \Vert_{\infty} (\kappa-\Theta(\varphi(a-h))]\Theta(\varphi(a-h)) \leq \Theta\left( \widehat{v}_1(\varphi;h)(a)\right) \leq [1-h \mu_-  (\kappa-\Theta(\varphi(a-h))] \Theta(\varphi(a-h))
$$
and since the map  $s\in [0,\kappa]\rightarrow   [1-h \mu_- (\kappa-s)] s$ is non decreasing for $h>0$ sufficiently small it follows that there exists $h_0>0$ depending only on $\kappa$ and  $\mu$  such that 
\begin{equation}\label{5.13}
0\leq \Theta\left( \widehat{v}_1(\varphi;h)(a)\right)\leq \kappa,\ \forall a\geq h,\ \forall h\in [0,h_0].
\end{equation}
For $0\leq a<h$ by using \eqref{5.11} we have 
$$
\begin{array}{llll}
\Theta\left( \widehat{v}_1(\varphi;h)(a)\right) &=& \Theta\left(F_0(\varphi)\right)\\
&=&\displaystyle \Theta\left(\int_0^{+\infty} \beta(a)\varphi(a)[\kappa-\Theta(\varphi(a))] da\right)\\
&=&\displaystyle\int_0^{a_\dagger} \beta(a) [\kappa-\Theta(\varphi(a))] \Theta\left(\varphi(a)\right)  da.
\end{array}
$$
Since the maximum of the map $s\in [0,\kappa]\rightarrow (\kappa-s) s$  is $\dfrac{\kappa^2}{4}$ we obtain that 
\begin{equation}\label{5.14}
0\leq \Theta\left( \widehat{v}_1(\varphi;h)(a)\right) \leq \dfrac{\kappa^2}{4} \int_0^{a_\dagger} \beta(a) da  ,\ \text{ if }  0\leq a <h,\ h>0
\end{equation}
Thus we see from \eqref{5.12} and \eqref{5.14} that 
\begin{equation}\label{5.15}
0\leq \Theta\left( \widehat{v}_1(\varphi;h)(a)\right) \leq  \kappa,\ \text{ if }  0\leq a <h,\ h>0.
\end{equation}
The result follows from \eqref{5.13} and \eqref{5.15}.
\end{proof}\\
We have the following result. 
\begin{theorem} \label{TH5.3}
Assume  that 
\begin{equation*}
\int_0^{a_\dagger} \beta(a) da \leq \dfrac{4}{\kappa}.
\end{equation*}
Then for each $u_0 \in L^p_+(\mathbb{R}_+,\mathbb{R}^n)$ with 
$$
0\leq \Theta(u_0(a))\leq \kappa,\ \text{ for a.e. } a\in \mathbb{R}_+
$$
there exists a unique continuous globally defined mild solution $t\in \mathbb{R}_+\rightarrow u(t,.)\in L^p_+(\mathbb{R}_+,\mathbb{R}^n)$ of \eqref{5.1}  such that 
\begin{equation}\label{5.16}
0\leq \Theta(u(t,a))\leq \kappa,\ \text{ for a.e.  } a\in \mathbb{R}_+,\ \forall t\geq 0.
\end{equation}
\end{theorem}
\begin{proof}
The existence of a maximally defined solution of \eqref{5.3}  satisfying \eqref{5.16} is direct application of Theorem \ref{TH4.6}. To obtain the global existence of the solution of \eqref{5.3}  it is enough to observe that 
$$
F\left( \begin{array}{cc}
0_{\mathbb{R}^n}\\
\varphi
\end{array}
\right)\leq \left( \begin{array}{cc}
\int_0^{+\infty} \beta(a)\kappa \varphi(a) da \\  
\mu(.)\kappa \varphi(.)
\end{array}
\right),\ \forall \left( \begin{array}{cc}
0_{\mathbb{R}^n}\\
\varphi
\end{array}
\right) \in X_{0+}:=\{0_{\mathbb{R}^n}\} \times L^p_+(\mathbb{R}_+,\mathbb{R}^n).
$$
and infer from \cite[Corollary 3.7]{Magal-Ruan09a}. The result follows by using the fact that system \eqref{5.1} coincides with \eqref{5.3} in $C$.
\end{proof}

\end{document}